\documentclass[11pt,a4paper,leqno]{amsart}

\usepackage{amsmath,amsfonts,amssymb,amsthm}
\usepackage[english]{babel}
\usepackage{xcolor}
\usepackage{datetime}
\usepackage{enumerate}
\usepackage[foot]{amsaddr}

\newtheorem{theorem}{Theorem}

\newtheorem{lemma}{Lemma}

\newtheorem{corollary}{Corollary}

\newtheorem{theo}{Theorem}[section]

\newtheorem{theor}[theo]{Theorem}
\newtheorem{lem}[theo]{Lemma}

\newtheorem{defin}{Definition}

\theoremstyle{definition}
\newtheorem{remark}[theo]{Remark}

\numberwithin{equation}{section}

\usepackage{xcolor}

\newcommand{\n}[1]{#1}

\def\R{\mathbb{R}}
\def\C{\mathbb{C}}
\def\D{\mathbb{D}}
\def\Z{\mathbb{Z}}

\def\cl{\mbox{\rm cl}}
\def\E{\mathbf{E}}
\def\V{\mathbf{V}}
\def\P{\mathbf{P}}

\def\K{\mathcal{K }}

\title[Large Powers, Khinchin families, Lagrangian distributions]{Large Powers asymptotics, Khinchin families and Lagrangian distributions}

\author[J.\,L. Fern\'{a}ndez]{Jos\'{e} L. Fern\'{a}ndez}
\address[Jos\'{e} L. Fern\'{a}ndez]{Departamento de Matem\'{a}ticas, Universidad Aut\'{o}noma de Madrid, Spain.}
\email{joseluis.fernandez@uam.es, jlfernandez@akusmatika.org}

\thanks{Research of J. L. Fern\'{a}ndez is supported by Fundaci\'on Akusmatika}
\author[V. J. Maci\'{a}]{V\'{\i}ctor J. Maci\'{a}}
\address[V\'{\i}ctor J. Maci\'{a}]{Departamento de Matem\'{a}ticas, Universidad Aut\'{o}noma de Madrid, Spain.}
\email{victor.macia@uam.es}
\thanks{Research of V. Maci\'{a} was partially funded by grant MTM2017-85934-C3-2-P2 of Ministerio de Econom\'ia y Competitividad of Spain and European Research Council Advanced Grant 834728.
}


\begin{document}

\begin{abstract}This paper delves on the versatility of the theory of Khinchin families for asymptotic estimation. We show that in combination with Local Central Limit theorems for lattice variables, Khinchin families furnish  a convenient and unified framework to deal with asymptotic results of the coefficients of large powers of power series.

We revisit in the present paper this classical theme from that point of view, obtaining clean new proofs and a number of new results.

Asymptotic results for the coefficients of solutions of Lagrange's equation fall naturally into this combined framework. We provide a direct proof of an extension of the  Otter and Meir-Moon asymptotic formula as well as asymptotic results for families of Lagrangian probability distributions.

\end{abstract}

\maketitle

	\parskip=1.5mm

\setcounter{tocdepth}{1}

\footnotesize
\normalsize

\section{Introduction}

We  present  a  unified approach to deal with \textit{the asymptotic behaviour of the coefficients of large powers of power series with nonnegative coefficients.}

The  framework of this approach arises from the use  of Local Central Limit Theorems  for individual lattice random variables (Theorem \ref{teor:integral de diferencia caracteristicas})
and also for continuous families of lattice random variables (Theorem
\ref{teor:integral de diferencia caracteristicas familias continuas}) as a tool within the theory of Khinchin families of random variables associated with power series; in particular and foremost by coupling Local Limit  Theorems with  Hayman's asymptotic formula (Theorem \ref{teor:local central limit})
and Hayman's Central Limit Theorem (Theorem \ref{teor:hayman asymptotic formula}).

If $f(z)=\sum_{k=0}^{\infty} a_k z^k$ is a power series with nonnegative coefficients and positive radius of convergence $R$,  its \textit{associated Khinchin family}  $(X_t)_{t \in (0,R)}$ is the family of probability power series distributions indexed by $t \in (0,R)$ and given by
$$\P(X_t=k)=\frac{a_k t^k}{f(t)}\, , \quad \mbox{for $t\in (0,R)$ and integer $k \ge 0$}\, .$$
Thus each $X_t$ of the family is a random variable which takes values in the nonnegative integers following  a \textit{power series distribution}; $\psi_t(z)=f(tz)/f(t)$ is the probability generating function of $X_t$.

Khinchin families are thus families of power series distributions associated with a single power series. The theory of Khinchin  families starts with Hayman in \cite{Hayman}, where the Central Limit Theorem for Hayman's power series appears and also with B{\'a}ez-Duarte in \cite{BaezDuarte}, where strongly Gaussian power series are introduced. The terminology of Khinchin families originates in Rosenbloom \cite{Rosenbloom}, where, for instance, the Wiman-Valiron theorem on entire functions is proved by applying Chebyshev's inequality to the Khinchin family of an entire function.
The theory of Khinchin families has been  developed  at length in \cite{K_uno} and \cite{K_dos}.

The most basic Khinchin families are associated with the Taylor series of $f(z)=e^z$, $f(z)=1/(1-z)$ and $f(z)=(1+z)$. For the exponential function $f(z)=e^z$, the variable $X_t$ of its Khinchin family is a Poisson variable with parameter $t>0$. For $f(z)=1/(1-z)$, the variable $X_t$, for $t \in (0,1)$,  is geometric with failure probability $t$, while for $f(z)=(1+z)$ the variable $X_t$, for $t \in (0,1)$, is Bernoulli with parameter $t/(1+t)$.

Some of the most interesting Khinchin families are associated with generating functions, ordinary, like the generating function of partitions $P(z)=\prod_{j=1}^\infty 1/(1-z^j)$ with radius of convergence $R=1$ or, exponential, like  the exponential generating function of labeled  trees $T(z)=\sum_{n=0}^\infty n^{n{-}2} z^n/n!$ with radius of convergence $R=1/e$, or the exponential generating function of sets of sets (or partitions of sets) $B(z)=e^{e^{z{-}1}}$, with $R=\infty$.

\medskip

Large powers arise while considering sums of independent copies of a random variables: if $\psi(z)$ is the probability generating function of a random variable $Y$ taking values in $\{0,1, \ldots\}$ and $Y_1, \ldots, Y_n$ are independent copies of $Y$ then
$$\P\Big(\sum_{j=1}^n Y_j=k\Big)=\textsc{coeff}_{[k]}(\psi(z)^n)\,,$$
and if $(X_t)_{t \in (0,R)}$  is the Khinchin family associated with some power series $f$, and $X_t^{(1)}, \ldots, X_t^{(n)}$ are $n$ independent copies of $X_t$ then
$$\begin{aligned}
\P\Big(\sum_{j=1}^n X_t^{(j)}=k\Big)&=\textsc{coeff}_{[k]}\Bigg(\bigg(\frac{f(tz)}{f(t)}\bigg)^n\Bigg)\\&=\textsc{coeff}_{[k]}(f^n(z))\, \frac{t^k}{f^{n}(t)}\, , \quad \mbox{for any $k \ge 0$ and $t \in (0,R)$}\, .\end{aligned}$$
Large powers  also appear in Combinatorics. For instance, if $\psi(z)$ is the generating function of an unlabelled combinatorial class $\mathcal{C}$, then  $\textsc{coeff}_{[k]}(\psi(z))$ is the number of the objects in the class with weight $k$, while $\textsc{coeff}_{[k]}(\psi(z)^n)$ counts the number of lists of length $n$ of objects of $\mathcal{C}$ with total weight $k$. For general background on (Analytic) Combinatorics we refer to the comprehensive treatise \cite{Flajolet}.

\medskip

Along this paper, we reserve $k$ to signify \textit{index of the coefficient} and $n$ to denote  \textit{the power to which a power series $f$ is raised.} The power $n$ tends to infinity. Different results appear depending on how $k$ behaves with $n$, or actually, on how $k/n$ behaves as $n \to \infty$. As we shall see the Local Central Limit theorem for continuous families is used to handle the large exponent $n$ while the Khinchin families, particularly, of Gaussian (or strongly Gaussian) power series are used to handle the coefficient index $k$. In \cite{K_uno}, \cite{K_dos}, and,  of course,  in \cite{Hayman} and also in \cite{BaezDuarte}, one can see Khinchin families in action  efficiently estimating  coefficients of power series, with no large  powers involved.

\medskip

The usual tools to obtain these  asymptotic formulas for large powers: the saddle-point approximation,  steepest descent or Laplace method, play no role in the approach followed in this paper, although it must be said that these tools  are somehow built into the Local Central Limit Theorems and into the Hayman class.

Gardy, in \cite{Gardy}, gives an excellent  presentation of large powers asymptotics for $k/n\asymp 1$ and $k/n \to 0$ as  $n \to \infty$, from the point of view of saddle-point approximation. See also, De Angelis  \cite{Angelis}, for a quite direct approach.  See also Daniels \cite{Daniels}, particularly \cite[Section 8]{Daniels}, and Good \cite{Good}, particularly \cite[Section 6]{Good}, for saddle point approximations pertaining to probability generating functions.

\medskip

For general information on power series distributions we refer to Forbes in \cite[Chapter 37]{Forbes} 
and also to Patil in \cite{Patil}.

\medskip

\noindent \textbf{Plan of the paper.} The ingredients of the  approach  of this paper are reviewed in Sections 2 and 3.  The basic theory of Khinchin families in Section \ref{section:Khinchin families}, and the Local Central Limit Theorem for continuous families of lattice variables in
Section \ref{section:local central limit continuous families}.

\medskip

As discussed in Section \ref{section:large powers}, the combination of these two tools give rise to  Hayman's coefficient  formulas. These formulas are then used in Sections
 \ref{section:k comp n}, \ref{section:k=o(n)},  \ref{section:n=o(k)} and \ref{section:coef h psi^n} to obtain large powers (as $n\to \infty$) asymptotic results, which  depend upon the behaviour of the coefficient index $k$ relative to the large power $n$.

\medskip

In Section \ref{section:meir-moon} we cast in the present framework of Local Central Limits and Khinchin families the asymptotic formula due to Otter, \cite{Otter} and Meir-Moon, \cite{MeirMoonOld} for  \textit{the coefficients of the solutions of Lagrange equations}, when  the data  power series has nonnegative coefficients.  The proof presented in Section \ref{section:meir-moon} of the Otter-Meir-Moon Theorem includes a  limit case due to Janson and  proved in \cite[Appendix]{Janson} by   quite a different method; see also \cite[Chapter 3]{Drmota} and also \cite[Section 2.3.1]{Minami}.

Section \ref{section:prob gen functions} discusses applications of large powers asymptotics to probability generating functions and to \textit{Lagrangian power series distributions} and \textit{Galton-Watson or cascade processes}. For background on Lagrangian distributions we refer to \cite{Sibuya} of Sibuya, Miyawaki  and Sumita, for a neat presentation of the basic theory of these  probability distributions and also to \cite{Consul} of Consul and Famoye, for a comprehensive treatment. We refer also to \cite{Minami} for some recent results.

Finally,  Appendix \ref{section:uniformly Gaussian} deals with the so-called uniformly Gaussian and uniformly Hayman Khinchin families which are needed exclusively in Section  \ref{section:n=o(k)}.
\medskip

\noindent \textbf{Notations. }

Expectation and variance of a random variable $X$ are generically denoted $\E(X)$ and $\V(X)$. Probability of an event $A$ is generically denoted by $ \P(A)$.

For random variables $X,Y$ the notation $X\overset{d}{=} Y$ signifies that $X$ and $Y$ have the same distribution: $\P(X\in B)=\P(Y\in B)$ for any Borel set $B\subset \R$.

For a sequence of random variables $(X_n)_{n \ge 1}$ and a random variable $Y$ we write $$X_n\stackrel{d}{\underset{n \to \infty}{\longrightarrow}} Y$$ to signify convergence in distribution.

A family $(Z_s)_{s\in I}$ indexed in a real interval $I$ is continuous in distribution if for any sequence $(s_n)_{n \ge 1}$ of points in $I$ which converges to a point $s^\star\in I$ it holds that $Z_{s_n}$ converges in distribution to $Z_{s^\star}$.

\smallskip

The abbreviations \textit{egf} and \textit{ogf} mean, respectively, exponential generating function of a labeled combinatorial class and ordinary generating function of a combinatorial class.
\smallskip

The unit disk in the complex plane $\C$ is denoted by $\D$. The open disk of center $a\in \C$ and radius $R$ is denoted  $\D(a,R)$, while its closure is denoted $\cl(\D(a,R))$.

For positive sequences $(a_n)_{n\ge 1}, (b_n)_{n\ge 1}$,  the notation  $a_n \sim b_n$,  as $n\to \infty$, means that $\lim_{n \to \infty} a_n/b_n=~1$.

If $\psi$ and $f$ are probability generating functions, with $\mathcal{L}(\psi,f)$ we denote the Lagrangian probability distribution  with generators $\psi$ and $f$, see Section \ref{section:lagrangian distributions}.

\section{Khinchin families}
\label{section:Khinchin families}

In this section we describe  the specific aspects of the theory of Khinchin families to be used in the present paper. For proofs, examples and a variety of applications we  refer to \cite{K_uno}, and also to \cite{K_dos}.

We denote by $\K$ the class of nonconstant power series
$$f(z)=\sum_{n=0}^\infty a_n z^n$$  with positive radius of convergence,  which have nonnegative Taylor coefficients and such that  $a_0>0$. Since $f \in \K$ is nonconstant, at least one coefficient other than $a_0$ is positive.

The \textit{Khinchin  family} of  such a  power series $f \in \K$ with radius of convergence $R>0$ is the family of random variables $(X_t)_{t \in [0,R)}$ with values in $\{0, 1, \ldots\}$ and with mass functions given by
$$
\P(X_t=n)=\frac{a_n t^n}{f(t)}\, , \quad \mbox{for each $n \ge 0$ and $t \in (0,R)$}\, .$$ For $t=0$, we define $X_0\equiv 0$. Notice that $f(t)>0$ for each $t \in[0,R)$.

Any   Khinchin family is continuous in distribution in the interval $[0,R)$. No hypothesis upon joint distribution of the variables $X_t$ is considered. Each $(X_t)_{t \in [0,R)}$  is a family of random variables and not a stochastic process.


\subsubsection{Shifted Khinchin families}

Sometimes it is convenient to consider power series $g(z)=\sum_{n=0}^\infty b_n z^n$ with radius of convergence $R>0$ and nonnegative coefficients with at least two positive coefficients, but which may have $g(0)=0$.  We say then that $g$ is in the  shifted class $\K^s$. Thus $g \in \K^s$ if there exists an integer $l \geq 0$ such that
\begin{align*}
	g(z)/z^l \in \K\,,
\end{align*}
i.e., if  there exists a power series  $f(z) \in \K$ with radius of convergence $R>0$ such that $g(z) = z^l f(z)$, for $z \in \D(0,R)$.

To $g\in \K^s$ we associate a Khinchin family $(Y_t)_{t\
	\in (0,R)}$ as above:
$$\P(Y_t=n)=b_n t^n/g(t)\, , \quad \mbox{for $n \ge 0$ and $t\in (0,R)$}\,.$$

If $(Z_t)_{t \in [0,R)}$ is the Khinchin family of $f \in \K$, then
\begin{align*}
	Y_t \stackrel{d}{=} Z_t+l\,,
\end{align*}
\medskip
including,  $Y_0 \equiv l$.

\subsection{Basic properties}

Along this section we let $f$ be a power series in $\K$ with radius of convergence $R>0$ and Khinchin family $(X_t)_{t \in [0,R)}$.

\subsubsection{Mean and variance functions}
For the mean and variance of $X_t$ we reserve the notation
$m_f(t)=\E(X_t)$ and $\sigma_f^2(t)=\V(X_t)$, for $t \in [0,R)$. In terms of the power series $f$, the  mean and the variance of $X_t$ may be  written as
$$m_f(t)=\frac{t f^\prime(t)}{f(t)}, \qquad \sigma_f^2(t)=t m^\prime(t)\, , \quad \mbox{for $t \in [0,R)$}\,.$$

For each $t \in (0,R)$, the variable $X_t$ is not a constant, and so $\sigma_f^2(t)>0$. Consequently,  $m_f(t)$ is strictly increasing in $[0,R)$, though, in general, $\sigma_f(t)$ is not increasing. We  denote \begin{equation}\label{eq:def Mf} M_f=\lim_{t \uparrow R} m_f(t)\,.\end{equation}

\

For $g \in \K^s$, with Khinchin family $(Y_t)$, we also write $m_g(t)=\E(Y_t)$ and $\sigma_g^2(t)=\V(Y_t)$. If $g(z)=z^l h(z)$, with $l\ge 1$ and $h\in \K$ we have that
\begin{equation*}
	m_{g}(t) = l+m_{h}(t)\quad
	\mbox{and} \quad
	\sigma_{g}^2(t) = \sigma^2_{h}(t).
\end{equation*}
The function $m_{g}$ is an increasing diffeomorphism from $[0,R)$ to $[l,l+M_{h})$.

\medskip

For finite radius of convergence $R >0$ and for any integer $k \ge 1$, the function $f(t)$ and its derivatives $f^{(k)}(t)$ are all increasing functions on the interval $[0,R)$. For $k\ge 1$, we denote with $f^{(k)}(R)$ the limit
$$f^{(k)}(R)\triangleq \lim_{t\uparrow R} f^{(k)}(t)\,,$$
including
$$f(R)\triangleq \lim_{t\uparrow R} f(t)\,;$$
these limits exist, although they could be $+\infty$.

\subsubsection{Range of the mean}\label{section:range of mean}
Since the mean function $m_f(t)$ is increasing, its range is given by $[0,M_f)$.

\noindent $\bullet$  The case  where $M_f=\infty$ is particularly relevant and quite general. If $M_f=\infty$, then $m_f(t)$ is a diffeomorphism  from $[0,R)$ onto $[0, +\infty)$ and, in particular, for each $n \ge 1$, there exists a unique $t_n\in (0,R)$ such that $m_f(t_n)=n$. These $t_n$ play an important role in Hayman's identity, Section \ref{section:haymans identity},  and in Hayman's asymptotic formula, Section \ref{section:haymans asymptotic formula}.

\noindent $\bullet$ The following Lemma \ref{lemma:char of Mf finite} describes the quite specific cases where $M_f<+\infty$; it is Lemma 2.2 of \cite{K_uno}.

\begin{lemma}\label{lemma:char of Mf finite}
	For $f(z)=\sum_{n=0}^\infty a_n z^n$ in $\mathcal{K}$ with radius of convergence $R>0$, we have $M_f<\infty$ if and only if \big($R<\infty$ and $\sum_{n=0}^\infty n a_n R^n <\infty$\big) or \big($R=\infty$ and $f$ is a polynomial\big).\end{lemma}

In the first case, we have that $M_f=\dfrac{\sum_{n=0}^\infty n a_n R^n}{\sum_{n=0}^\infty a_n R^n}$, while for a polynomial $f\in \mathcal{K}$, the second case,  we have that $M_f=\mbox{deg}(f) $.

\subsubsection{Extension of the Khinchin family if $M_f<\infty$ {\upshape{(}}and $R<\infty${\upshape{)}}}\label{section:extension of family} We assume now that $M_f<\infty$ (and $R<+\infty$).

Thus we have that $\sum_{n=0}^\infty n a_n R^n <\infty$ and also that $f(R)=\sum_{n=0}^\infty a_n R^n <\infty$. The power series $\sum_{n=0}^\infty  a_n z^n$ defines a continuous (actually, $C^1$) function on the whole closed disk $\cl(\D(0,R))$ which extends $f$ from the open disk $\D(0,R)$.  We let $f(R)\triangleq \sum_{n=0}^\infty a_n R^n=\lim_{t\uparrow R} f(t)$ and
$$\begin{aligned}f^\prime(R)&\triangleq\sum_{n=1}^\infty n a_n R^{n{-}1}&=\lim_{t\uparrow R} f^\prime(t),\\ f^{\prime\prime}(R)&\triangleq\sum_{n=2}^\infty n (n{-}1) a_n R^{n{-}2}&=\lim_{t\uparrow R} f^{\prime\prime}(t).\end{aligned}$$
In this case, we may extend the Khinchin family $(X_t)_{t\in [0,R)}$ of $f$ to $t \in [0,R]$ by defining the variable
$X_R$ by
$$\P(X_R=n)=\frac{ a_n R^n}{f(R)}\,, \quad \mbox{for each $n \ge 0$}\,.$$
The extended family $(X_t)_{t\in [0,R]}$ becomes continuous in distribution in the closed interval $[0,R]$. Observe that $X_R$ (like any other $X_t$, with $t \in (0,R)$) is nonconstant.

The variable $X_R$ has (finite) mean $\E(X_R)=R f^\prime(R)/f(R)\triangleq m_f(R)$ and variance
$$\V(X_R)=\frac{\sum_{n=0}^\infty n^2 a_n R^n}{f(R)}-\E(X_R)^2\,.$$
The variance $\V(X_R)$ is nonzero since $X_R$ is nonconstant, but it could be infinite. Actually, $\V(X_R)$ is finite if and only if $\sum_{n=0}^\infty n^2 a_n R^n<+\infty$ if and only if $f^{\prime\prime}(R)<+\infty$, and in any case
$$\V(X_R)=\lim_{t \uparrow R} \sigma_f^2(t)\,.$$

If $\V(X_R)$ is finite, we write $ \sigma_f^2(R)\triangleq\V(X_R)$.
\subsubsection{Normalization and characteristic functions}

For each $t \in (0,R)$, we denote by  $\breve{X}_t$ the normalization of $X_t$:
$$\breve{X}_t=\frac{X_t-m_f(t)}{\sigma_f(t)},\, \quad \mbox{for $t \in (0,R)$}\, .$$
The characteristic function of $X_t$ may be written in terms of the power series $f$ as:
$$\E(e^{\imath \theta X_t})=\frac{f(te^{\imath \theta})}{f(t)}\, , \quad \mbox{for $t\in (0,R)$ and $\theta \in \R$}\,, $$
while for its normalized version $\breve{X}_t$ we have
$$\E(e^{\imath \theta \breve{X}_t})=\E(e^{\imath \theta X_t/\sigma_f(t)}) e^{-\imath \theta m_f(t)/\sigma_f(t)}\, , \quad \mbox{for $t\in (0,R)$ and $\theta \in \R$}\,,$$
and so,
$$\big|\E(e^{\imath \theta \breve{X}_t})\big|=\big|\E(e^{\imath \theta X_t/\sigma_f(t)})\big|\, , \quad \mbox{for $t\in (0,R)$ and  $\theta \in \R$}\,.$$

\

If $M_f<\infty$ (and $R<\infty$), we may define $\breve{X}_R=(X_R-m_f(R))/\sigma_f(R)$, with the understanding that $\breve{X}_R\equiv 0$, if $\sigma_f(R)=+\infty$.  Recall that $\sigma_f(R)>0$. The extended family $(\breve{X}_t)_{t \in (0,R]}$ of normalized variables is continuous (in distribution).

\subsubsection{Scaling}\label{section:scale}

For any power series $g(z)=\sum_{n=0}^\infty a_n z^n$ in $\K$ of radius $R>0$, we denote
$$Q_g\triangleq\gcd\{n\ge 1: a_n \neq 0\}=\lim_{N \to \infty} \gcd\{1\le n \le N: a_n \neq0\}\,.$$
If $Q_g>1$, then we can write $g(z)=f(z^{Q_g})$ for a certain companion power series  $f \in \mathcal{K}$ with  has radius of convergence $R^{Q_g}$. Observe that $Q_f=1$.

Let  $(Y_t)_{t \in [0, R)}$ be the Khinchin family of $g$ and $(X_t)_{t \in [0,R^{Q_g})}$ be the Khinchin family of $f$. We have that
\begin{align*}
	Y_t \stackrel{d}{=}Q_g \cdot X_{t^{Q_g}}, \quad \text{ for any } t \in (0,R)).
\end{align*}

The mean and variance functions of $g$ and $f$ are related by
$$\begin{aligned}
	m_g(t)=Q_g \cdot m_f(t^{Q_g})\, , \\
	\sigma^2_g(t)=Q_g^2 \cdot \sigma^2_f(t^{Q_g})\, ,
\end{aligned}\qquad \mbox{for $t\in (0,R)$}\,.
$$

\subsubsection{Fulcrum $F$ of $f$}\label{section:auxiliary F} A power series $f$ in $\K$ does not vanish on the real interval $[0,R)$, and so, it does not vanish in a simply connected region containing that interval.  We may consider $\ln f$, a branch of the logarithm of $f$ which is real on $[0,R)$, and the  function $F$, called the \textit{fulcrum} of $f$, which is  defined and holomorphic in a region containing $(-\infty, \ln R)$ and it is given by
$$F(z)=\ln f(e^z)\, .$$

If $f$ does not vanish anywhere in the disk $\D(0,R)$, then the fulcrum $F(z)$ of $f$ is defined everywhere in the whole half plane $\Re z< \ln R$.

In general, the fulcrum $F$ of $f$ is  defined and holomorphic in the band-like region $\{s+\imath \theta: s<\ln R, |\theta|< \sqrt{2}/\sigma_f(e^s)\}$. See Section 2.1.5 of \cite{K_uno}.

For $t=e^s$, we have $m_f(t)=F^\prime(s)$ and $\sigma_f^2(t)=F^{\prime\prime}(s)$, for $s <\ln R$.

\subsubsection{Products} Let $f$ and $g$ be  power series in $\K$ with  radius of convergence $R$ and $S$ and associated Khinchin families $(X_t)_{t \in (0,R)}$ and $(Y_t)_{t \in (0,S)}$, respectively. The product $h=fg$ is also in $\K$. Let $T=\min\{R,S\}$, and let $(Z_t)_{t \in (0,T)}$ be the Khinchin family associated with $h$.

Then for each $t \in (0,T)$ the law of $Z_t$ is that of \textit{a sum of independent copies} of $X_t$ and $Y_t$:
$$\P(Z_t=k)=\sum_{j=0}^k \P(X_t=j)\P(Y_t=k{-}j)\, , \quad \mbox{for each $k \ge 0$}\,.$$

In particular, if $(X_t)_{t \in (0,R)}$  is the Khinchin family associated with some power series $f$, and if we let $X_t^{(1)}, \ldots, X_t^{(n)}$ denote  $n$ independent copies of $X_t$ then
$$
\P\Big(\sum_{j=1}^n X_t^{(j)}=k\Big)=\textsc{coeff}_{[k]}\big(f^n(z)\big)\, \frac{t^k}{f^{n}(t)}\, , \quad \mbox{for any $k \ge 0$ and $t \in (0,R)$}\, .$$

\subsubsection{Basic families and some examples}\label{section:basic families and some examples} For the exponential $f(z)=e^z$ we have $m_f(t)=t$ and $\sigma_f^2(t)=t$, for each $t>0$, as it should since the variable $X_t$ of the associated Khinchin family follows a Poisson distribution with parameter $t$. For $f(t)=1/(1-t)$ whose Khinchin family comprises geometric variables, one has $m_f(t)=t/(1-t)$ and $\sigma_f^2(t)=t/(1-t)^2$, for $t \in (0,1)$.

For the exponential generating function of partitions of sets $B(z)=e^{e^z{-}1}$, one has $m_B(z)=t e^t$ and $\sigma_B^2(t)=t(1{+}t)e^t$.

For the generating function of partitions of integers $P(z)=\prod_{j=1}^\infty 1/(1-z^j)$, one has that
\begin{equation*}
m_P(t)=\sum_{j=1}^\infty \frac{j t^j}{1-t^j} \quad \mbox{y}\quad \sigma_P^2(t)=\sum_{j=1}^\infty \frac{j^2 t^j}{(1-t^j)^2}\,, \quad \mbox{for any $t\in (0,1)$}\,,\end{equation*}
with the asymptotic approximation
\begin{equation}\label{eq:asymptotics media var partition function}
m_P(e^{-s}) \sim \frac{\zeta(2)}{s^2} \quad \mbox{y}\quad \sigma^2_P(e^{-s})\sim \frac{2 \zeta(2)}{s^3}\,, \quad \mbox{as $s\downarrow 0$}\,.
\end{equation}
See \cite[Section 6.1]{K_uno}.
\subsection{Hayman's identity}\label{section:haymans identity}

Let $f(z)=\sum_{n=0}^\infty a_n z^n$ be a power series in $\K$. Cauchy's formula for the coefficient $a_n$ in terms of the characteristic function of its Khinchin  family $(X_t)_{t \in [0,R)}$ reads
$$a_n=\frac{f(t)}{2\pi t^n }\int_{|\theta|<\pi} \E(e^{\imath \theta X_t})e^{-\imath \theta n} \, d \theta\, , \quad \mbox{for each $t \in (0,R)$ and $n\ge 1$}\, .$$

In terms of the characteristic function of the normalized variable $\breve{X}_t$, it becomes
\begin{equation}\label{eq:expresion integral general de a_n previa a hayman}\begin{aligned}a_n=\frac{f(t)}{2\pi t^n \sigma_f(t)}\int_{|\theta|<\pi\sigma_f(t)} \E(e^{\imath \theta \breve{X}_t})&e^{-\imath \theta (n-m_f(t))/\sigma_f(t)} \, d \theta\, , \\&\quad \mbox{for each $t \in (0,R)$ and $n \ge 1$}\, .\end{aligned}\end{equation}

\noindent $\bullet$ If $M_f=\infty$, we may take for each $n\ge 1$ the (unique) radius $t_n \in (0,R)$  so that $m_f(t_n)=n$, to write
\begin{equation}\label{eq:haymans formula}a_n=\frac{f(t_n)}{2\pi t_n^n \sigma_f(t_n)}\int_{|\theta|<\pi\sigma_f(t_n)} \E(e^{\imath \theta \breve{X}_{t_n}}) \, d \theta\, , \quad \mbox{for each $n \ge 1$}\, ,\end{equation}
which we call \textit{Hayman's identity}.

Although this identity \eqref{eq:haymans formula} is just Cauchy's formula with an appropriate choice of radius, it  encapsulates the saddle point method.

\medskip

As we now check, if the Khinchin family $(X_t)$ may be extended to the closed interval $[0,R]$, then an analogous Hayman's identity is available for $t=R$.

\noindent  $\bullet$ If $M_f<\infty$ and ($R<\infty$) the power series $f$ extends continuously to the closed disk $\cl(\D(0,R))$ and the Khinchin family $(X_t)_{t \in [0,R)}$ extends to the closed interval $[0,R]$. See Section \ref{section:range of mean} and the notations therein.

We may write
$$a_n=\frac{f(R)}{2\pi R^n }\int_{|\theta|<\pi} \E(e^{\imath \theta X_R})e^{-\imath \theta n} \, d \theta\, , \quad \mbox{for $n\ge 1$}\, .$$

If, moreover, $\sigma_f^2(R)<\infty$, i.e., if $\sum_{n=0}^\infty n^2 a_n R^n<\infty$, then $\breve{X}_R$ given by $$\breve{X}_R=\frac{X_R{-}m_f(R)}{\sigma_f(R)}$$ is well defined and we also have that
\begin{equation}\label{eq:expresion integral general de a_n previa a hayman si Mf finito}a_n=\frac{f(R)}{2\pi R^n \sigma_f(R)}\int_{|\theta|<\pi\sigma_f(R)} \E(e^{\imath \theta \breve{X}_R})\, e^{-\imath \theta (n-m_f(R))/\sigma_f(R)} \, d \theta\, , \quad \mbox{for $n \ge 1$}\, .\end{equation}

\subsection{Moments of a Khinchin family near 0}\label{section:basic moments near 0}

Let $f(z)=\sum_{n=0}^\infty a_n z^n$ be a power series in $\K$ with radius of convergence $R>0$. Let $(X_t)_{t \in[0,R)}$ be its Khinchin family and  $m_f(t)$ and $\sigma_f^2(t)$ be, respectively, the mean and variance of $X_t$, for $t \in [0,R)$.

Assume first that $a_1>0$. In this case, we have that
$$m_f(t)=\frac{a_1}{a_0} t+O(t^2)  \quad \text{and} \quad \sigma_f^2(t)=\frac{a_1}{a_0} t+O(t^2) \, , \quad \mbox{as $t \downarrow 0$}\,, $$
and,
besides, that
$$\E(|X_t-m_f(t)|^3)=\frac{a_1}{a_0} t+O(t^2) \quad \text{and} \quad \sqrt{t} \, \E(|\breve{X}_t|^3)=\sqrt{\frac{a_0}{a_1}}  +O(t)\, , \quad \mbox{as $t \downarrow 0$}\, .$$

In general, if for $k \ge 1$, we have $a_k>0$, but $a_j=0$, for $1\le j < k$, then
$$m_f(t)=\frac{k a_k}{a_0} t^k+O(t^{k{+}1})  \quad \text{and} \quad \sigma_f^2(t)=\frac{k^2 a_k}{a_0} t^k+O(t^{k{+}1})\, , \quad \mbox{as $t \downarrow 0$}\,, $$
and,
besides, that
$$\E(|X_t-m_f(t)|^3)=\frac{k^3 a_k}{a_0} t^k+O(t^{k{+}1})\quad \text{and} \quad t^{k/2} \, \E(|\breve{X}_t|^3)=\sqrt{\frac{a_0}{a_k}} +O(t)\, , \quad \mbox{as $t \downarrow 0$}\, .$$

\subsection{Gaussian Khinchin families}\label{section:gaussian Khinchin families}

We introduce next Gaussian and strongly Gaussian power series.
\begin{defin}
	A power series $f $ in $\K$ of radius of convergence $R>0$ and its  Khinchin family $(X_t)_{t \in [0,R)}$ are termed \textit{Gaussian} if
	$\breve{X}_t$ converges in distribution to the standard normal, as $t\uparrow R$, or, equivalently, if
	$$\lim_{t \uparrow R} \E(e^{\imath \theta \breve{X}_t})=e^{-\theta^2/2}\, , \quad \mbox{for each $\theta \in \R$}\,.$$\end{defin}

\medskip

\begin{defin}\label{defin:strongly gaussian} A power series $f \in \K$ and its  Khinchin family $(X_t)_{t \in [0,R)}$ are termed  \textit{strongly Gaussian} if the following two conditions are satisfied
	\begin{equation*}\lim_{t \uparrow R} \sigma_f(t)=+\infty \quad \mbox{and} \quad \lim_{t \uparrow R} \int\limits_{|\theta|<\pi \sigma_f(t)}\Big|\E(e^{\imath \theta \breve{X}_t})-e^{-\theta^2/2}\Big| \, d\theta=0\, .\end{equation*}\end{defin}

This notion of strongly Gaussian power series was introduced by B\'{a}ez-Duarte in \cite{BaezDuarte}.
Strongly Gaussian power series are Gaussian, see Theorem \ref{teor:local central limit}.

\

We refer to \cite{K_uno} and \cite{K_dos} for examples, for criteria to verify whether  a power series is Gaussian or  strongly Gaussian, and for a number of applications (asymptotic formulas of coefficients) of these concepts for set constructions in Analytic Combinatorics, including partitions.

We just mention (see \cite{K_uno}) that
\begin{itemize} \item the exponential $f(z)=e^z$ is strongly Gaussian and thus Gaussian,
	\item polynomials are \textit{not} Gaussian. In fact, if $(X_t)_{t \in [0,\infty)}$ is the Khinchin family of a polynomial of degree $N$, then, as $t \uparrow \infty$,  $X_t$ tends in distribution to the constant $N$ and $\breve{X}_t$ tends in distribution to $0$.
\end{itemize}

\begin{lemma}\label{lema:gaussian Mf inf} If $f$ is a Gaussian power series, then $M_f=+\infty$\end{lemma}
\begin{proof}Assume that $M_f<+\infty$. We use the characterization of $M_f<+\infty$ given in Lemma \ref{lemma:char of Mf finite}.
	If the radius of convergence $R$ of $f$ is $R=+\infty$, then $f$ is a polynomial, which is not Gaussian as we just have mentioned.
	
	If $R<\infty$, then (see Section \ref{section:range of mean}) we may extend the Khinchin family to include $t=R$ with $X_R$. Thus $X_t$ tends in distribution to $X_R$ as $t \uparrow R$. If $\sigma_f(R)<\infty$, then $\breve{X}_t$ tends in distribution to $\breve{X}_R=(X_R-m_f(R))/\sigma_f(R)$, while if $\sigma_f(R)=\infty$, then $\breve{X}_t$ tends in distribution to the constant $\breve{X}_R\equiv 0$. In both cases, the limit $\breve{X}_R$ is a discrete random variable.
\end{proof}
\subsection{Hayman's Central Limit Theorem}

Strongly Gaussian power series satisfy Theorem \ref{teor:local central limit} below, which is both a local and global Central Limit Theorem. This result, for Hayman's admissible power series, comes from \cite{Hayman}, but see \cite{K_uno} for a proof for the more general case of strongly Gaussian power series.

\begin{theorem}[Hayman's  Central Limit Theorem] \label{teor:local central limit} If $f(z)=\sum_{n=0}^\infty a_n z^n $ in $\K$ is a strongly Gaussian power series with radius of convergence $R>0$, then
	\begin{equation}\label{eq:local central limit}\lim_{t \uparrow R} \sup\limits_{n \in \Z} \Big|\frac{a_n t^n}{f(t)} \sqrt{2\pi}\sigma_f(t)-e^{-(n-m_f(t))^2/(2\sigma_f^2(t))}\Big|=0\,.\end{equation}
	{\upshape{(}}In this statement $a_n=0$, for $n <0$.{\upshape{)}}
	
	Besides,$$
	\lim_{t \uparrow R}\P( \breve{X}_t \le b)=\Phi(b)\,, \quad \mbox{for every $b \in \R$}\, .$$
	And so,   $\breve{X}_t$ converges in distribution towards the standard normal  and $f$ is Gaussian.
\end{theorem}

\smallskip

In other terms, Theorem \ref{teor:local central limit} claims that if for each $t\in[0,R)$ we define the set  $$\mathcal{V}_t=\Big\{\frac{n-m_f(t)}{\sigma_f(t)}: n \in \Z\Big\}\, ,$$
then
$$
\lim_{t \uparrow R} \sup\limits_{x \in \mathcal{V}_t} \Big|\P(\breve{X}_t =x)\, \sigma_f(t)-\frac{1}{\sqrt{2\pi}} e^{-x^2/2}\Big|=0\, .$$

\subsection{Asymptotic formula of coefficients of strongly Gaussian power series}\label{section:haymans asymptotic formula}

The coefficients of strongly Gaussian power series obey a neat asymptotic formula:

\begin{theorem}[Hayman's asymptotic formula]\label{teor:hayman asymptotic formula} If $f(z)=\sum_{n=0}^\infty a_n z^n$ in $\K$ is strongly Gaussian then
	\begin{equation}\label{eq:asymptotic formula coeffs}a_n \sim  \frac{1}{\sqrt{2\pi}} \frac{f(t_n)}{t_n^n \sigma_f(t_n)}\, , \quad \mbox{as $n \to \infty$}\, .\end{equation}
\end{theorem}
In the asymptotic formula above,  $t_n$ is given by $m_f(t_n)=n$, for each $n \ge 1$, which are uniquely defined because $M_f=+\infty$, see Lemma \ref{lema:gaussian Mf inf}. This asymptotic formula follows readily from Theorem \ref{teor:local central limit}, or alternatively, from Hayman's formula \eqref{eq:haymans formula} and strong gaussianity.

\

Actually, if $\omega_n$ is a good approximation of $t_n$, in the sense that
$$\lim_{n \to \infty} \frac{m_f(\omega_n)-n}{\sigma_f(\omega_n)}=0\, ,$$
then
\begin{equation}\label{eq:formula asintotica coeficientes de Hayman bis}a_n \sim \frac{1}{\sqrt{2\pi}}\frac{f(\omega_n)}{\sigma_f(\omega_n) \, \omega_n^n}\, ,\quad \mbox{as $n \to \infty$}\,.
\end{equation}
In general, precise expressions for the $t_n$ are rare, since inverting $m_f(t)$ is usually  complicated. See, for instance the examples in Section \ref{section:basic families and some examples}. But, fortunately, in practice, one can do with a certain  asymptotic approximation due to B\'{a}ez-Duarte, \cite{BaezDuarte}, the Substitution Theorem,  which we now describe.

Suppose that $ f \in \K$ is strongly Gaussian. Assume that $\widetilde{m}(t)$ is continuous and monotonically increasing to $+\infty$ in $[0,R)$ and that $\widetilde{m}(t)$ is a good approximation of $m_f(t)$ in the sense that
\begin{equation}\label{eq:condition baez-duarte}
\lim_{t \uparrow R} \frac{m_f(t)-\widetilde{m}(t)}{\sigma_f(t)}=0\,.\end{equation}
Let  $\tau_n$ be  defined by $\widetilde{m}(\tau_n)=n$, for each $n \ge 1$.

\begin{theorem}[Substitution]\label{teor:baez-duarte} With the notations above, if $f(z)=\sum_{n=0}^\infty a_n z^n$ in $\K$ is strongly Gaussian and \eqref{eq:condition baez-duarte} is satisfied, then
$$a_n \sim  \frac{1}{\sqrt{2\pi}} \frac{f(\tau_n)}{\tau_n^n \sigma_f(\tau_n)}\, , \quad \mbox{as $n \to \infty$}\, .$$
\end{theorem}
This follows readily from \eqref{eq:formula asintotica coeficientes de Hayman bis}. Besides, if  $\widetilde{\sigma}(t)$ is such  that $\sigma_f(t) \sim  \widetilde{\sigma}(t)$ as $t \uparrow R$, we may further write
\begin{equation}\label{eq:formula de hayman-baez-duarte}a_n \sim  \frac{1}{\sqrt{2\pi}} \frac{f(\tau_n)}{\tau_n^n \widetilde{\sigma}(\tau_n)}\, , \quad \mbox{as $n \to \infty$}\, .\end{equation}

\

The exponential function $f(z)=e^z$ is strongly Gaussian; Hayman's asymptotic formula for $e^z$ is just Stirling formula.

The ordinary generating function of partitions of integers:  $P(z)=\prod_{j=1}^\infty (1/(1-z^j)$ is strongly Gaussian; Hayman's asymptotic formula for $P(z)$ as in Theorem \ref{teor:baez-duarte} and  involving the approximations collected in \eqref{eq:asymptotics media var partition function} becomes the Hardy-Ramanujan  asymptotic formula for the number of partitions of an integer. See \cite{BaezDuarte}, \cite{K_uno} and \cite{K_dos}.

The exponential generating function of partitions of sets: $B(z)=e^{e^z{-}1}$ is strongly Gaussian; Hayman's asymptotic formula for $B(z)$ as in Theorem \ref{teor:baez-duarte} and appealing to asymptotics of the Lambert function becomes the Moser-Wyman asymptotic formula for the Bell numbers, or number of partitions of a set of size $n$. See \cite{K_uno} and \cite{K_dos}.

\subsection{Hayman class}\label{section:Hayman class}

The class of Hayman consists of power series $f$ in $\K$ which satisfy some concrete and verifiable conditions which imply that $f$ is strongly Gaussian, see Theorem \ref{teor:hayman implica baez} below.

\medskip

\begin{defin}A power series $f\in \K $ is in the \textit{Hayman class} {\upshape{(}}or is Hayman-admissible or just $H$-admissible{\upshape{)}} if
	\begin{equation}
		\label{eq:condicion de varianza en Hayman}\mbox{\rm variance condition:}\quad \lim_{t \uparrow R} \sigma_f(t)=\infty\,,\end{equation} and for a certain function $h: [0,R)\rightarrow (0, \pi]$, which we refer to as a  {\mbox{\textrm{cut}}} between a \textrm{major} arc and a \textit{minor} arc, the following conditions are satisfied.
	\begin{align}\label{eq:h mayor caracteristica}\mbox{\rm major arc:} \quad
		&\lim_{t \uparrow R}\sup_{|\theta|\le h(t)\, \sigma_f(t)} \bigg|\E(e^{\imath \theta \breve{X}_t})e^{\theta^2/2}-1\bigg|=0\, , \\
		\label{eq:h menor caracteristica}\mbox{ \rm minor arc:} \quad
		&\lim_{t\uparrow R}\sigma_f(t) \, \sup_{h(t)\sigma_f(t)\le |\theta| \le \pi \sigma_f(t)} |\E(e^{\imath \theta \breve{X}_t})|=0\, .
\end{align}\end{defin}

\smallskip

Some authors include within the Hayman class   power series with a finite number of negative coefficients. This is not the case in the present paper.

\smallskip

For $f$ in the Hayman class,  the characteristic  function of $\breve{X}_t$ is uniformly approximated by  $e^{-\theta^2/2}$ in  the major arc, while it is uniformly $o(1/\sigma(t))$ in the minor arc.

\medskip

\begin{theorem}\label{teor:hayman implica baez} Power series in the Hayman class are strongly Gaussian.\end{theorem}

See \cite{K_uno} for a proof of Theorem \ref{teor:hayman implica baez}. We may consider membership in the Hayman class as a criterion for strong gaussianity.

\section{Local Central Limit Theorem and continuous families}
\label{section:local central limit continuous families}

To obtain asymptotic results for the coefficients of large powers of power series with nonnegative coefficients we will write down formulae (to be presented in Section \ref{section:hayman and large powers}) for the coefficients of the large powers in terms of the associated Khinchin families and Hayman's identities.

To handle the Hayman's identities of a  Khinchin family we shall use  the Local Central Limit Theorem for lattice variables in some integral form. This is Theorem \ref{teor:integral de diferencia caracteristicas}, as applied to a single lattice variable, and Theorem \ref{teor:integral de diferencia caracteristicas familias continuas}, which is a version of the Local Central Limit Theorem for continuous families of lattice variables. We provide a detailed  proof of the latter result which  requires some preliminary material which we will discuss  in Section \ref{section:continuous families of lattice variables}.

\subsection{Approximation and bound of characteristic functions}

For a random variable $Y$ and integer $n\ge 1$, we have for $\E\big(e^{\imath \theta Y/\sqrt{n}}\big)^n$ the following well known approximation,  Theorem \ref{teor:convergencia uniforme TLC}, and bound, Theorem \ref{teor:cota uniforme caracter TLC}, involving the third moment of $Y$.

\begin{theorem}\label{teor:convergencia uniforme TLC} There exists a constant $C>0$, such that for any random variable $Y$ verifying that $\E(Y)=0$, $\E(Y^2)=1$ and $\E(|Y|^3)<\infty$, it holds that
	$$
	\Big|\E\big(e^{\imath \theta Y/\sqrt{n}}\big)^n-e^{-\theta^2/2}\Big|\le C \frac{\E(|Y|^3)}{\sqrt{n}}\,, \quad \mbox{if $|\theta|\le \dfrac{\sqrt{n}}{\E(|Y|^3)}$ and $n \ge 1$}\,.$$
\end{theorem}
See, for instance, \cite[Lemma 6.2, Chapter 7]{Gut}. In particular,
\begin{equation}\label{eq:teorema limite central basico}\lim_{n \to \infty} \E\big(e^{\imath \theta Y/\sqrt{n}}\big)^n=e^{-\theta^2/2}\,, \quad \mbox{for each $\theta \in \R$}\,,\end{equation}
which is the Central Limit Theorem for sums of independent identically distributed copies of the variable $Y$ with the extra assumption of finite third absolute moment.

\begin{theorem}\label{teor:cota uniforme caracter TLC}  For  any random variable $Y$ such that $\E(Y)=0$, $\E(Y^2)=1$ and  $\E(|Y|^3)<\infty$, we have that
	$$\Big|\E\big(e^{\imath \theta Y/\sqrt{n}}\big)^n\Big|\le e^{-\theta^2/3}, \quad \mbox{for any $\theta$ such that  $|\theta|\le \dfrac{\sqrt{n}}{4\E(|Y|^3)}$ and $n \ge 1$}\,.$$
\end{theorem}
See, for instance, step 1 in the proof of \cite[Lemma 6.2, Chapter 7]{Gut}.

%
%
%
%
%
%

\subsection{Lattice distributed variables}

A random variable $Z$ is said to be \textit{lattice distributed} or a \textit{lattice random variable} if  $Z$ is nonconstant and for a certain  $a \in \R$ and  $h>0$ one has that $\P(Z\in a+h\Z)=1$. The largest  $h$ such that $\P\big(Z\in a+h\Z\big)=1$ is called the  \textit{gauge} of the lattice variable  $Z$. For gauge $h$, the variable  $(Z-a)/h$ takes integer values and has gauge 1.

\smallskip

Let  $f(z)=\sum_{n=0}^\infty a_n z^n$ be a power series in $\K$. Let $\mathcal{I}_f=\{n\ge 1:a_n \neq 0\}$. Recall from  Section \ref{section:scale} the notation  $Q_f=\mbox{gcd}\{\mathcal{I}_f\}$.

If $Q_f=1$, then each $X_t$ of the Khinchin family of $f$ is a lattice variable of gauge 1 and each $\breve{X}_t$ is a lattice variable of gauge $1/\sigma_f(t)$. The condition $Q_f=1$ is equivalent to the requirement that for each $d\ge 1$ there is $n \in \mathcal{I}_f$ such that $d \nmid n$.

In general, if $Q_f \ge 1$, then  each $X_t$ of the Khinchin family of $f$ is a lattice variable of gauge $Q_f$ and each $\breve{X}_t$ is a lattice variable of gauge $Q_f/\sigma_f(t)$.

\begin{lemma}\label{lema:caracteristica de reticular} If $Z$ is a lattice random variable with gauge  $h$, then
	$$|\E(e^{\imath \theta Z})|<1\,, \quad \mbox{ for  $0<\theta <2\pi/h$}\,.$$
	In particular, for each $0<\delta< \pi/h$, there exists  $\omega=\omega_\delta<1$ such that $$|\E(e^{\imath \theta Z})|\le \omega\, , \quad \mbox{for  $|\theta- \pi/h|\le \delta$}\,.$$
\end{lemma}
See \cite[Section 3.5]{Durrett}. The second part of the statement follows simply from continuity of the characteristic function.

%

The next result  is a standard inversion lemma whereas the mass function of a variable $U$ is expressed in terms of its characteristic function.

\begin{lemma} \label{lema:inversion variable reticular} Let $U$ be a lattice random variable with gauge $h$,  so that for certain  $a\in \R$ we have that $\P(U \in a+h \Z)=1$.
	Then
	$$\P(U=a+hk)=\frac{h}{2\pi}\int_{-\pi/ h}^{\pi/ h}\E(e^{\imath \theta U}) e^{-\imath \theta (a+kh)}\, d\theta\, , \quad \mbox{for each  $k \in \Z$}\, .$$
\end{lemma}

In other terms, for each attainable value $x$ of the variable $U$ we have that $$\P(U=x)=\frac{h}{2\pi}\int_{-\pi/ h}^{\pi/ h}\E(e^{\imath \theta U}) e^{-\imath \theta x}d\theta\, .$$

\begin{proof} We may assume that $a=0$ and $h=1$. For each $k \in \Z$, let  $p_k=\P(U=k)$. Then,   we have that $\sum_{k \in \Z} p_k=1$ and
	$$\E(e^{\imath \theta U})=\sum_{k \in \Z} p_k e^{\imath k \theta}\, ,\quad \mbox{for each $\theta \in \R$}\,,$$
	and therefore, that
	$p_k=\dfrac{1}{2\pi}\displaystyle\int_{-\pi}^\pi \E(e^{\imath \theta U}) e^{-\imath k \theta} d\theta\, , \quad \mbox{for each  $k \in \Z$}\,.$
\end{proof}

\subsection{Local Central Limit Theorem for lattice variables: integral form}

\begin{theorem}\label{teor:integral de diferencia caracteristicas} Let $Z$ be a lattice random variable with gauge  $h$ and such that  $\E(Z)=0$ and $\E(Z^2)=1$. Then
	$$\lim_{n \to \infty} \int\limits_{|\theta|\le \pi \sqrt{n}/h} \Big|\E(e^{\imath \theta Z/\sqrt{n}})^n-e^{-\theta^2/2}\Big|\, d\theta=0\,.$$
\end{theorem}

From the convergence of this integral, the usual statement of the Local Central Limit Theorem follows, see Section \ref{section:local central theorem usual}. In fact, Theorem \ref{teor:integral de diferencia caracteristicas} is what is actually shown (as an intermediate step)  in the usual proofs of the Local Central Limit Theorem, see,  for instance, \cite[Section 43]{Gnedenko} and the proof of \cite[Theorem 3.5.3]{Durrett}.

We give a proof below with the extra assumption that $\E(|Y|^3)<\infty$ anticipating and as preparation for the corresponding result, Theorem \ref{teor:integral de diferencia caracteristicas familias continuas}, for continuous families of lattice variables.

\begin{proof} The Central Limit Theorem, as in equation \eqref{eq:teorema limite central basico},  gives  that, as $n \to \infty$,  the integrand converges towards 0 for each $\theta \in \R$.
	
	Denote $\tau\triangleq 1/(2\E(|Z|^3)$. Theorem  \ref{teor:cota uniforme caracter TLC} provides us with  the bound
	$$|\E(e^{\imath \theta Z/\sqrt{n}})^n|\le e^{-\theta^2/6}\, , \quad \mbox{si $|\theta|\le \tau \sqrt{n}$}\,.$$
	
	Thus, dominated convergence implies that
	$$(\flat)\quad \lim_{n \to \infty} \int_{|\theta|\le \tau \sqrt{n}} \Big|\E(e^{\imath \theta Z/\sqrt{n}})^n-e^{-\theta^2/2}\Big|\, d\theta=0\,.$$

	\smallskip
	Lemma \ref{lema:caracteristica de reticular} gives, in particular,  a bound $\omega\in (0,1)$ such that
	$$|\E(e^{\imath\theta Z})|\le \omega\, , \quad \mbox{si $ \tau \le |\theta|\le \pi/h$}\, ,$$
	and so that
	$$|\E(e^{\imath\theta Z/\sqrt{n}})^n|\le \omega^n\, , \quad \mbox{si $ \tau \sqrt{n}\le |\theta|\le \pi\sqrt{n}/h$}\, .$$
	
	The bound
	$$\int\limits_{\tau \sqrt{n}\le |\theta|\le \pi\sqrt{n}/h}\Big|\E(e^{\imath \theta Z/\sqrt{n}})^n-e^{-\theta^2/2}\Big|\, d\theta\le
	\int\limits_{\tau \sqrt{n}\le |\theta|\le \pi\sqrt{n}/h}\Big(\omega^n +e^{-\theta^2/2}\Big)\, d\theta\, ,$$
	tends to 0 as $n \to \infty$. This together with $(\flat)$ gives the result.\end{proof}

\subsubsection{Local Central Limit Theorem}\label{section:local central theorem usual} Let  $Y$ be a random variable such that  $\E(Y)=0$ and  $\E(Y^2)=1$. Let $Y_1, Y_2, \ldots$ be a sequence of independent random variables each one of them  distributed like  $Y$.

Assume that $Y$ is a lattice variable with gauge  $h$ and  that $\P(Y\in a+h \Z)=1$, for some fixed  $a \in \R$.

\smallskip
For each $n \ge 1$, we denote $S_n=(Y_1+\dots+Y_n)/\sqrt{n}$.
The variable $S_n$ is likewise a lattice variable, with gauge  $h/\sqrt{n}$.

Let  $\mathcal{V}_n$ denote the set of values that the variable $S_n$ can take, i.e.,~$\mathcal{V}_n=\{a \sqrt{n} +k h /\sqrt{n}: k \in \Z\}$.

\smallskip

Applying Lemma \ref{lema:inversion variable reticular} to the variable $S_n$ and using that $\E(e^{\imath \theta S_n})=\E(e^{\imath \theta Y/\sqrt{n}})^n$, we have that
$$\P(S_n=x)=\frac{h}{2\pi\sqrt{n}}\int_{-\pi \sqrt{n}/h}^{\pi\sqrt{n}/h} \E(e^{\imath  \theta Y/\sqrt{n}})^n e^{-\imath \theta x}\, d\theta\, ,\quad \mbox{for every $x \in \mathcal{V}_n$}\,.$$

Since
$$\frac{1}{\sqrt{2\pi}}e^{-x^2/2}=\frac{1}{2\pi}\int_\R e^{-\theta^2/2} e^{-\imath \theta x}\, d\theta\, , \quad \mbox{for every $x \in \R$}\,,$$
we deduce that
$$\begin{aligned}\sup_{x \in \mathcal{V}_n} \Big|\frac{\sqrt{n}}{h}\P(S_n=x)-\frac{1}{\sqrt{2\pi}}e^{-x^2/2}\Big|&\le \frac{1}{2\pi}\int\limits_{|\theta|\le \pi \sqrt{n}/h}\Big|\E(e^{\imath \theta Y/\sqrt{n}})^n-e^{-\theta^2/2}\Big|\, d\theta\\&+
	\frac{1}{2\pi} \int\limits_{|\theta|>\pi\sqrt{n}/h} e^{-\theta^2/2}\, d\theta\, .\end{aligned}
$$

Theorem \ref{teor:integral de diferencia caracteristicas} allows us to conclude from the bound above that

\begin{corollary}[Local Central Limit Theorem]\label{teor:limite central local}
	With the notations above,
	$$\lim_{n\to \infty}
	\sup_{x \in \mathcal{V}_n} \left|\frac{\sqrt{n}}{h}\P(S_n=x)-\frac{1}{\sqrt{2\pi}}e^{-x^2/2}\right|=0\, .$$
\end{corollary}

\

\subsection{Continuous families of lattice variables}\label{section:continuous families of lattice variables}
A family $(Z_s)_{s \in [a,b]}$ of random variables indexed in the real interval $[a,b]$ is said to be \textit{continuous} if it is continuous in distribution in the sense that if a sequence $(s_n)_{n \ge 1}$ with $s_n\in [a,b]$ converges to $s\in[a,b]$ then  $Z_{s_n}$ converges to $Z_s$ in distribution.

The family $(Z_s)_{s \in [a,b]}$ is said to be \textit{bounded} if $\sup_{t\in[a,b]} \E\left(\left|Z_t\right|\right)<+\infty$.

\begin{lem}\label{lemma:continuous and bounded}
	If  the family $(Z_s)_{s \in [a,b]}$ is continuous and bounded, then the function $$(s,\theta) \in [a,b]\times \R \longrightarrow \E(e^{\imath \theta Z_s }) \in \D\,$$
	is continuous.
\end{lem}

We shall apply this lemma to the normalized version $\breve{X}_t$ of a Khinchin family of $f \in \K$ when $t$ is restricted to a closed interval $[a,b]\subset (0,R)$. For this family continuity is clear and regarding boundedness, observe that $\E(|\breve{X}_t|)\le \E(|\breve{X}_t|^2)^{1/2}=1$, for $t\in [a,b]$.  If $M_f<+\infty$ (and $R<+\infty$), and if, besides,  $\sigma_f(R)<\infty$ (see Section  \ref{section:extension of family}), then $(\breve{X}_t)$ is continuous and bounded for $t\in [a,R]$, (including $t=R$) for any $a \in (0,R)$.

\smallskip

The (un-normalized) Khinchin family $(X_t)$ of any $f$ in $\K$ with radius of convergence $R$ is continuous and bounded in any interval $[0,b]$ with $0 < b<R$. If $M_f<+\infty$ (and $R<+\infty$) the family $(X_t)_{t \in [0,R]}$ extended to the closed interval $[0,R]$ is continuous and bounded; observe that $\E(|X_t|)=\E(X_t)\le M_f$, for $t \in [0,R]$.

\begin{proof}
	Take two sequences $(\theta_n)_{n \geq 1} \subseteq \R$ and $(s_n)_{n \geq 1} \subseteq [a,b]$ which, respectively,  converge to $\theta \in \R$ and $s\in [a,b]$. First we write
	\begin{align}\label{eq:continua_characteristic}
		\E(e^{\imath\theta_n Z_{s_n}}) = \E(e^{\imath\theta Z_{s_n}})+\E(e^{\imath\theta_n Z_{s_n}}-e^{\imath\theta Z_{s_n}}).
	\end{align}
	
	By virtue of Levy's convergence theorem, the first term in the right-hand side of (\ref{eq:continua_characteristic}) converges, as $n \to \infty$, to $\E(e^{\imath\theta Z_s})$. We  have the bound
	\begin{align*}
		\left|\E(e^{\imath\theta_n Z_{s_n}}-e^{\imath\theta Z_{s_n}})\right|  \leq \left|\theta_n-\theta\right| \E(|Z_{s_n}|), \quad \mbox{for $n \ge 1$}\,,
	\end{align*}
	which  follows  from the numerical  inequality: $|e^{ix}-e^{iy}| \leq |x-y|$, for all $x,y \in \R$.
	
	Now using that $(Z_s)_{s \in [a,b]}$ is bounded,  we conclude that
	$$(s,\theta) \in [a,b]\times \R \longrightarrow \E(e^{\imath \theta Z_s }) \in \D\, ,$$
	is a continuous function.
\end{proof}

\smallskip

Let $(Z_s)_{s \in [a,b]}$ be a continuous family of lattice variables. For each $s \in [a,b]$, let $h(s)$ denote the gauge of $Z_s$. The gauge function $h$ is upper semicontinuous:
$$\limsup_{u \to s} h(u)\le h(s)\, , \quad \mbox{for each $s \in [a,b]$}\,.$$ In general, the gauge of a continuous family is not a continuous function. For instance, for $s \in [0,1/2]$, let the variable $Z_s$ take the values   $0,1/2$ and $1$ with respective probabilities $1/2~{-}s, 2s, 1/2~{-}s$. The family  $(Z_s)$ is continuous, and $h(s)=1/2$, for $s \neq 0$, but $h(0)=1$.

\smallskip

The next lemma is a counterpart of Lemma \ref{lema:caracteristica de reticular} for continuous and bounded  families; it follows from Lemma \ref{lema:caracteristica de reticular} by continuity and compactness.

\begin{lem}\label{lema:caracteristica de reticular familia continua}  Let  $(Z_s)_{s \in [a,b]}$ be a continuous and bounded family of lattice random variables, with {\underline{continuous}} gauge function $h(s)$.
	
	Let  $H>0$, be such that  $h(s)\le H$, for each $s \in [a,b]$.
	
	Then,
	$$|\E(e^{\imath \theta Z_s})|<1, \quad \mbox{if $s \in [a,b]$ and $0<|\theta|<2\pi/h(s)$}\, .$$
	
	In particular, for each $0<\tau<\pi/H$, there exists  $\omega_\tau <1$ such that $$|\E(e^{\imath \theta Z_s})|\le \omega_\tau, \quad \mbox{if $s \in [a,b]$ and $\tau\le|\theta|\le\pi/h(s)$}\, .$$
\end{lem}

\smallskip

The next theorem is  a uniform analogue of Theorem \ref{teor:integral de diferencia caracteristicas} for continuous families of lattice variables:

\begin{theor}\label{teor:integral de diferencia caracteristicas familias continuas} Let  $(Z_s)_{s \in[a,b]}$ be a continuous family of lattice variables with gauge function $h(s)$ continuous in $[a,b]$ and such that
	for each  $s \in [a,b]$ we have that $\E(Z_s)=0$, $\E(Z_s^2)=1$ and $\E(|Z_s|^3)\le \Gamma$, for a certain constant  $\Gamma>0$.
	
	\smallskip
	
	Then,
	$$\lim_{n \to \infty} \sup_{s \in [a,b]} \int_{-\pi \sqrt{n}/h(s)}^{\pi \sqrt{n}/h(s)} \Big|\E(e^{\imath \theta Z_s/\sqrt{n}})^n-e^{-\theta^2/2}\Big|\, d\theta=0\,.$$
\end{theor}

\begin{remark} The hypothesis  $\sup_{s\in[a,b]} \E(|Z_s|^3)<\infty$ may be replaced by the assumption that $\sup_{s\in[a,b]} \E(|Z_s|^{2+\delta})<\infty$, for some $\delta>0$,  or even by assuming that the family $(|Z_s|^2)_{s\in [a,b]}$ is uniformly integrable.\end{remark}

\begin{proof} First, we pay attention to the range of integration $|\theta|\le \sqrt{n}/(4\Gamma)$.
	
	Theorem  \ref{teor:cota uniforme caracter TLC}  gives us that
	$$\Big|\E\big(e^{\imath \theta Z_s/\sqrt{n}}\big)^n\Big|\le e^{-\theta^2/3}, \quad \mbox{if $s\in[a,b]$ and $|\theta|\le \dfrac{\sqrt{n}}{4\Gamma}$}\,,$$
	while Theorem \ref{teor:convergencia uniforme TLC} gives us that
	$$
	\Big|\E\big(e^{\imath \theta Z_s/\sqrt{n}}\big)^n-e^{-\theta^2/2}\Big|\le C  \frac{\Gamma }{\sqrt{n}}\,, \quad \mbox{if $s \in [a,b]$ and $|\theta|\le \dfrac{\sqrt{n}}{\Gamma}$}\,.$$
	
	Therefore, for any sequence $(s_n)_{n \ge 1}$ extracted from $[a,b]$ we deduce,  from dominated convergence, that
	$$\lim_{n \to \infty}\int_{-\sqrt{n}/(4\Gamma)}^{ \sqrt{n}/(4\Gamma)} \Big|\E(e^{\imath \theta Z_{s_n}/\sqrt{n}})^n-e^{-\theta^2/2}\Big|\, d\theta=0\,,$$
	and, therefore, that
	$$(\star) \quad \lim_{n \to \infty} \sup_{s \in [a,b]} \int_{-\sqrt{n}/(4\Gamma)}^{ \sqrt{n}/(4\Gamma)} \Big|\E(e^{\imath \theta Z_s/\sqrt{n}})^n-e^{-\theta^2/2}\Big|\, d\theta=0\,.$$
	
	\smallskip
	
	Fix $0<J<H$ such that  $J\le h(s)\le H$, for each $s \in [a,b]$. Take  $\tau>0$ such that $\tau<1/(4\Gamma)$ and $\tau <\pi/H$.
	
	Since, for every $s\in[a,b]$, we have that  $\E(|Z_s|)\le \E(|Z_s|^2)^{1/2}=1$, the family $(Z_s)_{s\in [a,b]}$ is bounded. Thus, Lemma  \ref{lema:caracteristica de reticular familia continua}  shows that there exists
	$\omega_\tau<1$ such that
	$$|\E(e^{\imath \theta Z_s})|\le \omega_\tau, \quad \mbox{if $s \in [a,b]$ and $\tau\le|\theta|\le\pi/h(s)$}\, ,$$
	and so  such that
	$$|\E(e^{\imath \theta Z_s/\sqrt{n}})^n|\le \omega_\tau^n, \quad \mbox{if $s \in [a,b]$ and $\tau \sqrt{n}\le|\theta|\le\pi \sqrt{n}/h(s)$}\, .$$
	
	Therefore, for each $s \in [a,b]$ we have that
	$$
	\int\limits_{\tau \sqrt{n}\le |\theta| \le \pi \sqrt{n}/ h(s)} \Big|\E(e^{\imath \theta Z_s/\sqrt{n}})^n-e^{-\theta^2/2}\Big|\, d\theta \le \omega_\tau^n \pi\sqrt{n}/{J}+\int\limits_{|\theta| \ge \tau \sqrt{n}} e^{-\theta^2/2} d \theta\, .$$
	Consequently,
	$$
	(\star\star)\quad \lim_{n \to \infty} \sup_{s\in[a,b]}\int\limits_{\tau \sqrt{n}\le |\theta| \le \pi \sqrt{n}/ h(s)} \Big|\E(e^{\imath \theta Z_s/\sqrt{n}})^n-e^{-\theta^2/2}\Big|\, d\theta=0\, .$$
	Combining $(\star)$ and $(\star\star)$ we obtain the result since $\tau<1/(4\Gamma)$.\end{proof}

\section{Coefficients of large powers}\label{section:large powers}

We now turn our attention to the  study of the asymptotic behavior of coefficients of large powers of functions $\psi$ in the class $\mathcal{K}$: the behaviour of the coefficient
$
\textsc{coeff}_{[k]}(\psi^n(z))$, as $n\to \infty$.

The objective of this section is to establish appropriate expressions for the Hayman's identities for the coefficients  $\textsc{coeff}_{[k]}(\psi^n(z))$ of (large) powers in terms of the Khinchin family of $\psi$: Lemma \ref{lemma:formula maestra para grandes potencias} and Lemma \ref{lemma:formula maestra para grandes potencias Mf finito}.

In forthcoming sections we will apply these lemmas to study the behaviour of the coefficient $\textsc{coeff}_{[k]}(\psi^n(z))$ under a variety of assumptions upon the joint behaviour of the index  $k$ and of the power $n$.  At a later stage, in Section \ref{section:coef h psi^n}, we will consider  the asymptotic behavior of the coefficients of $h(z)\psi^n(z)$ where $h(z)\in \mathcal{K}$.

\smallskip

\begin{remark}
	The results which follow  about large powers actually cover the general situation of $\psi$ with nonnegative coefficients, and not just $\psi \in \K$. To see this we may reason as follows.
	
	If $\psi$ has nonnegative coefficients and it is not in $\K$, then one of the two following possibilities occurs: 1) $\psi$ is a constant or  a monomial like $\psi(z)=b z^m$, for some $b \neq 0$ and integer $m \ge 1$, which are trivial situations as coefficients and large powers are concerned,  2) $\psi(0)=0$ and $\psi$ has at least two nonnegative coefficients, in which case, $\psi \in \K^s$ and for some integer $l \ge 1$ we have that $\varphi(z)\triangleq \psi(z)/z^l \in \K$, and thus $\textsc{coeff}_{[k]}(\psi^n)=\textsc{coeff}_{[k-nl]}(\varphi^n)$, for $k\ge nl$.
\end{remark}

\smallskip

Henceforth,  $\psi(z)=\sum_{j=0}^\infty b_j z^j$ will denote a power series in $\mathcal{K}$ with radius of  convergence $R > 0$, and from now on we let $(Y_t)_{t \in[0,R)}$ denote its Khinchin family.

\smallskip

We reserve  $k$ to denote index of coefficient and $n$ to denote power of $\psi$. In all asymptotic discussions below the power $n$ tends to $\infty$ (large powers),  while  the index $k$ could tend to $\infty$ in such a way that either $k \asymp n$ (including the case when $k/n$ tends to a finite nonzero limit) or $k=o(n)$ (including the possibility that $k$ could remain fixed) or $n=o(k)$ (which would require, as we shall see, that the power series $\psi$ is uniformly Gaussian). These three cases are discussed in Sections \ref{section:k comp n}, \ref{section:k=o(n)} and \ref{section:n=o(k)}, respectively.

%
%

\subsection{Auxiliary function $\phi$} Recall the notation of Section \ref{section:scale}:
$$Q_f=\gcd\{n\ge 1: a_n \neq 0\}=\lim_{N \to \infty} \gcd\{1\le n \le N: a_n \neq0\}$$
for any power series $f(z)=\sum_{n=0}^\infty a_n z^n$ in $\mathcal{K}$.

If $Q_\psi>1$, then $\psi(z)=\phi(z^{Q_\psi})$ for a certain auxiliary power series  $\phi \in \mathcal{K}$ which  has radius of convergence $R^{Q_\psi}$. Observe that $Q_\phi=1$ and that
$$\textsc{coeff}_{[kQ_\psi]}(\psi^n(z))=\textsc{coeff}_{[k]}(\phi^n(z^{Q_\psi})),$$
while for $q$ not a multiple of $Q_\psi$, we have that
$$\textsc{coeff}_{[q]}(\psi^n(z))=0\, .$$

\smallskip

Denote by $(Z_t)_{t \in (0,R^{Q})}$ the Khinchin family associated with this auxiliary function $\phi$, then, see Section \ref{section:scale},
\begin{align*}
	Y_t \stackrel{d}{=}Q_\psi \cdot Z_{t^{Q_\psi}}, \quad \text{ for any } t \in (0,R).
\end{align*}

The mean and variance functions of $\psi$ and $\phi$ are related by
$$
m_\psi(t)=Q_\psi \cdot m_\phi(t^{Q_\psi})\quad \mbox{and} \quad
\sigma^2_\psi(t)=Q_\psi^2 \cdot \sigma^2_\phi(t^{Q_\psi})\, .
$$

For each $t \in (0,R)$, the variable $Y_t$ is a lattice random variable with gauge $Q_\psi$. Likewise, the  normalized variable $\breve{Y}_t$ is a lattice random variable with gauge $h(t) = Q_\psi/\sigma_\psi(t)$.

\smallskip

In subsequent analysis, we shall obtain asymptotic formulae for $\textsc{coeff}_{[k]}(\psi(z)^n)$ first under the assumption that $Q_\psi=1$, and then in the general case when $Q_\psi\ge 1$, by considering the auxiliary power series $\phi$,  with $Q_\phi=1$, and translating the results back to $\psi$.

\subsection{Hayman's identities and large powers}\label{section:hayman and large powers}

We may express  $\textsc{coeff}_{[k]}(\psi(z)^n)$ in terms of the characteristic function of the normalized variables $(\breve{Y}_t)_{t \in [0,R)}$ by Hayman's identity as follows.

\begin{lem}\label{lemma:formula maestra para grandes potencias} With the notations above,
	\begin{equation} \label{eq:formula maestra para grandes potencias}
		\begin{aligned}\textsc{coeff}_{[k]}(\psi(z)^n)&=
			\frac{1}{2\pi} \frac{\psi^n(t)}{t^k}\frac{1}{\sqrt{n}}\frac{1}{\sigma_\psi(t)}\int\limits_{|\theta|\le \pi \sigma_\psi(t)  \sqrt{n}} \E\big(e^{\imath \theta\breve{Y}_t/\sqrt{n}}\big)^n e^{\imath(n m_\psi(t)-k)\theta/(\sigma_\psi(t)\sqrt{n})}d\theta\, ,\\
	\end{aligned}\end{equation}
	for any index $k\ge 1$, power $n \ge 1$ and  for all $t \in (0,R)$.
\end{lem}

\begin{proof} For $t \in (0,R)$, Cauchy's integral formula gives that
	\begin{align*}
		\textsc{coeff}_{[k]}(\psi(z)^n) &= \frac{1}{2\pi \imath} \int_{|z| = t} \frac{\psi(z)^{n}}{z^{k+1}}dz= \frac{1 }{2\pi } \frac{\psi(t)^n}{t^{k}} \int_{|\theta| < \pi} \frac{\psi(te^{\imath  \theta})^n}{\psi(t)^n}e^{-\imath\theta k}d\theta \\
		&= \frac{1 }{2\pi } \frac{\psi(t)^n}{t^{k}}  \int_{|\theta|< \pi} \E(e^{\imath\theta Y_{t}})^ne^{-\imath \theta k} d\theta \\
		&= \frac{1}{2\pi} \frac{\psi^n(t)}{t^k}\frac{1}{\sqrt{n}}\frac{1}{\sigma_\psi(t)}  \int_{|\theta| < \pi \sigma_\psi(t)\sqrt{n}} \E(e^{\imath\theta \breve{Y}_t/\sqrt{n}})^{n}e^{\imath(n m_\psi(t)-k)\theta/(\sigma_\psi(t)\sqrt{n})} d\theta.
	\end{align*}
\end{proof}

Observe that the integral expression of \eqref{eq:formula maestra para grandes potencias} is greatly simplified if the radius $t$ is such that $(\star) \ \ n \, m_\psi(t)=k$. Whenever possible we will select and use that value of $t$ such that $(\star)$ holds exactly or at least approximately.

In the next sections, we shall use the formula \eqref{eq:formula maestra para grandes potencias} to study the asymptotic behaviour of $\textsc{coeff}_{[k]}(\psi^n(z))$, as $n\to \infty$, while $k\asymp n$, $k/n \to 0$ or $k/n\to \infty$.

\smallskip

A minor variation of the proof of Lemma  \ref{lemma:formula maestra para grandes potencias} above gives that if $H(z)$ is a function holomorphic in $\D(0,R)$, then,  with the notations above, we have that
\begin{equation} \label{eq:formula maestra para grandes potencias funcion h}
	\begin{aligned}&\textsc{coeff}_{[k]}(H(z)\psi(z)^n)\\&=
		\frac{1}{2\pi} \frac{\psi^n(t)}{t^k}\frac{1}{\sqrt{n}}\frac{1}{\sigma_\psi(t)}\int\limits_{|\theta|\le \pi \sigma_\psi(t)  \sqrt{n}}H\big(t e^{\imath \theta/(\sigma_\psi(t)\sqrt{n})}\big) \E\big(e^{\imath \theta\breve{Y}_t/\sqrt{n}}\big)^n e^{\imath(n m_\psi(t)-k)\theta/(\sigma_\psi(t)\sqrt{n})}d\theta\, ,\\
\end{aligned}\end{equation}
for any index $k\ge 1$, power $n \ge 1$ and  $t \in (0,R)$.

\medskip

If $M_\psi<\infty$ (and $R<\infty$), then $\psi$ extends to be continuous in $\cl(\D(0,R))$ and the Khinchin family $(Y_t)_{t \in [0,R)}$ extends to the closed interval by adding $Y_R$ given by
$$\P(Y_R=n)=b_n R^n /\psi(R)\,, \quad \mbox{for $n \ge 0$}\,.$$
See Section \ref{section:range of mean} and Section \ref{section:haymans identity} and the notations therein. In this case, we can take $t=R$ to obtain the following formula for the coefficients of  $\psi^n$.

\begin{lem}\label{lemma:formula maestra para grandes potencias Mf finito} With the notations above, if $R<\infty$ and  $M_\psi=m_\psi(R)<\infty$ and $\sigma_\psi^2(R)<\infty$, then
	\begin{equation} \label{eq:formula maestra para grandes potencias Mf finito}
		\begin{aligned}&\textsc{coeff}_{[k]}(\psi(z)^n)\\&=
			\frac{1}{2\pi} \frac{\psi^n(R)}{R^k}\frac{1}{\sqrt{n}}\frac{1}{\sigma_\psi(R)}\int\limits_{|\theta|\le \pi \sigma_\psi(R)  \sqrt{n}} \E\big(e^{\imath \theta\breve{Y}_R/\sqrt{n}}\big)^n e^{\imath(n m_\psi(R)-k)\theta/(\sigma_\psi(R)\sqrt{n})}d\theta\, ,\\
	\end{aligned}\end{equation}
	for any index $k\ge 1$ and  power $n \ge 1$.
\end{lem}

\section{Index $k$ and power $n$ are comparable}\label{section:k comp n}

In this section we discuss the asymptotic behaviour of $\textsc{coeff}_{[k]}(\psi^n(z))$ when the index $k$ and the power $n$ are such that $ k\asymp n$, as $n \to \infty$.

The plan is to analyse the integral term in the Hayman's identities  of Section \ref{section:hayman and large powers} when $ k\asymp n$, as $n \to \infty$,  by invoking the Local Central Limits Theorems of Section \ref{section:local central limit continuous families}.

\smallskip

Along this section we assume that,  for certain $A,B$, fixed, $0<A<B$, the index $k$ and the power $n$ are such that $A\le k/n\le B$.

\smallskip

To deal with this case, we assume additionally that $M_\psi>B$. This being the case,  we denote $a=m_\psi^{-1}(A)$ and $b=m_\psi^{-1}(B)$.

\smallskip

Assume first that $Q_\psi=1$; a restriction to be lifted shortly in Section \ref{section:general Q}. The family of random variables  $\breve{Y}_t$, where $t \in [a,b]$, is a continuous family of lattice random variables with gauge function $h(t) = 1/\sigma_\psi(t)$; see Section \ref{section:continuous families of lattice variables}. In particular, $\max\limits_{t \in[a,b]} \E(|\breve{Y}_t|^3)<+\infty$.

\smallskip

For each $n \ge 1$, define $\tau_n$  by $m_\psi(\tau_n) = k/n$. This choice is possible because $M_{\psi} > B$. For each $n \geq 1$, we have $\tau_n \in [a,b]$.

Taking $t = \tau_n$ in Hayman's identity \eqref{eq:formula maestra para grandes potencias}, the integral term in there simplifies to
$$
I_n\triangleq \int\limits_{|\theta|\le \pi \sigma_\psi(\tau_n)  \sqrt{n}} \E\big(e^{\imath\theta \breve{Y}_{\tau_n} /\sqrt{n}}\big)^n d\theta\, .$$
Since $\sigma_\psi(\tau_n)=1/h(\tau_n)$, Theorem \ref{teor:integral de diferencia caracteristicas familias continuas} gives that
\begin{equation}\label{eq:limite fuerte cuando k como n} \lim_{n \to \infty}\int\limits_{|\theta|\le \pi \sigma_\psi(\tau_n)  \sqrt{n}} \Big|\E\big(e^{\imath\theta \breve{Y}_{\tau_n} /\sqrt{n}}\big)^n -e^{-\theta^2/2}\Big|d\theta=0\, .\end{equation}
In particular, since $\min_{t \in[a,b]} \sigma_\psi(t) >0$, we conclude that
$\lim_{n \to \infty} I_n=\sqrt{2\pi}$,
and, thus, that
$$
\textsc{coeff}_{[k]}(\psi(z)^n)\sim
\frac{1}{\sqrt{2\pi}} \frac{\psi^n(\tau_n)}{\tau_n^k}\frac{1}{\sqrt{n}}\frac{1}{\sigma_\psi(\tau_n)}\, , \quad \mbox{as $n \to \infty$}.$$

\

\subsection{General $Q_\psi$}\label{section:general Q}

We now lift
the assumption that $Q_\psi=1$. To simplify notation, write $Q=Q_\psi\ge 1$. Consider the auxiliary function $\phi$, so that $\psi(z)=\phi(z^Q)$.

\smallskip

Recall that $m_\psi(t)=Q m_\phi(t^Q)$, let $A^\prime=A/Q$ and $B^\prime=B/Q$ and observe that  $M_\phi=M_\psi/Q>B/Q=B^\prime$.

Define $\tau_n^\prime$  by $m_\phi(\tau_n^\prime)=k^\prime/n$, for $k^\prime$ such that  $A^\prime < k^\prime/n<B^\prime$.

\smallskip

Since $Q_\phi=1$, we have that
$$(\natural) \quad \textsc{coeff}_{[k^\prime]}(\phi^n(z))\sim \frac{1}{\sqrt{2\pi}}\frac{\phi(\tau_n^\prime)^n}{(\tau_n^\prime)^{k^{\prime}}} \frac{1}{\sqrt{n}}\frac{1}{\sigma_\phi(\tau_n^\prime)}\,, \quad \mbox{as $n \to \infty$}\, .$$

\smallskip

Let $\tau_n=(\tau_n^\prime)^{1/Q}$. Take $k=k^\prime Q$, observe that $A<\dfrac{k}{n}<B$ and that
$$m_\psi(\tau_n)=Q\, m_\phi(\tau_n^\prime)=\frac{k^\prime Q}{n}=\frac{k}{n}\, ,$$
and
$$\sigma_\psi(\tau_n)=Q\, \sigma_\phi(\tau_n^\prime)\, .$$

Since $\textsc{coeff}_{[k]}(\psi^n)=\textsc{coeff}_{[k^\prime]}(\phi^n)$, translating $(\natural)$ into  terms of $\psi$, we get that
$$\begin{aligned} \textsc{coeff}_{k}(\psi^n(z))&\sim \frac{Q_\psi}{\sqrt{2\pi}}\frac{\psi(\tau_n)^n}{\tau_n^{k}} \frac{1}{\sqrt{n}}\frac{1}{\sigma_\psi(\tau_n)}\,,\\& \quad \mbox{as $n \to \infty$ while $A<k/n<B$ and $k$ is a multiple of $Q$}\, .\end{aligned}$$

\begin{theo}\label{teor:asymptotic powers with k comp n}
	For $0<A<B<M_\psi$, we have that
	$$
	\textsc{coeff}_{k}(\psi^n(z))\sim \frac{Q_\psi}{\sqrt{2\pi}}\frac{\psi(\tau_n)^n}{\tau_n^{k}} \frac{1}{\sqrt{n}}\frac{1}{\sigma_\psi(\tau_n)},$$
	as $n \to \infty$, while $A\le {k}/{n} \le B$ and $k$ is a multiple of $Q_\psi$, and where $\tau_n$ is given uniquely by $m_\psi(\tau_n)={k}/{n}$.
\end{theo}

\subsection{With further information on $k/n$}

We assume now that  for some $L$ such that $0<L <+\infty$ and $L \le M_\psi$, we have that $ k/n \to L$, as $n \to \infty$. With that information, we may  sharpen the  asymptotic formula for $\textsc{coeff}_{[k]}(\psi(z)^n)$ of Theorem \ref{teor:asymptotic powers with k comp n}.

Assume that $k/n\to L$, where $0<L<+\infty$ and $L \le M_\psi$, and, in fact, in such a way that for some $\omega \in \R$
$$\frac{k}{n}=L+\omega \frac{1}{\sqrt{n}}+o\Big(\frac{1}{\sqrt{n}}\Big)\, , \quad \mbox{as $n\to \infty$}\,.$$
or, equivalently, that
\begin{equation}\label{eq:lim L and omega}\lim_{n \to \infty} \frac{nL-k}{\sqrt{n}}=-\omega\, .\end{equation}

\smallskip

\noindent $\bullet$ \textit{Suppose  that $L <M_\psi$}. Choose $\tau \in (0,R)$ such that $m_\psi(\tau)=L$. Observe that now we have a fixed value of  $\tau$, and no  $\tau_n$ varying with $n$.

Assume first that $Q_\psi=1$. If we choose $t = \tau$ in formula \eqref{eq:formula maestra para grandes potencias} and invoke the Local Central Limit Theorem \ref{teor:integral de diferencia caracteristicas} we readily find that
$$\begin{aligned}\textsc{coeff}_{[k]} (\psi(z)^n) &\sim \frac{1}{2\pi} \frac{\psi^n(\tau)}{\tau^k} \frac{1}{\sqrt{n}} \frac{1}{\sigma_\psi(\tau)}\int_\R e^{\imath \omega \theta /\sigma_\psi(\tau)} e^{-\theta^2/2}\, d\theta\\
	&=\frac{e^{-\omega^2/(2\sigma_f^2(\tau))}}{\sqrt{2\pi}\,\sigma_f(\tau)} \, \frac{1}{\sqrt{n}}\frac{\psi^n(\tau)}{\tau^k}\, ,\, \mbox{as  $n\to \infty$}.\end{aligned}
$$

An argument  much like the one in Section \ref{section:general Q} allows us to deduce the general case $Q_\psi\ge1$ from the case $Q_\psi=1$.

\begin{theo}\label{teor:large powers k/n to L} If $L<M_\psi$ and $$\lim_{n \to \infty} \frac{nL-k}{\sqrt{n}}=-\omega\,,$$ then
	$$\textsc{coeff}_{[k]} (\psi(z)^n) \sim Q_\psi\,\frac{e^{-\omega^2/(2\sigma_\psi^2(\tau))}}{\sqrt{2\pi}\,\sigma_\psi(\tau)} \, \frac{1}{\sqrt{n}}\frac{\psi^n(\tau)}{\tau^k}\, ,\, \mbox{as  $n\to \infty$}.$$
	where $\tau$ is given by $m_\psi(\tau)=L$  and $k$ is a multiple of $Q_\psi$.
\end{theo}

The case where $k=n{-}1$ and thus $L=1$ (and $\omega=0$) shall be used in the discussion of the Otter-Meir-Moon Theorem, Theorem \ref{teor:Otter-Meir-Moon}, on an asymptotic formula for the coefficients of the solutions of Lagrange's equation.

\medskip

\noindent $\bullet$ \textit{Suppose $L=M_\psi$ and $R<+\infty$. } Observe that $M_\psi<\infty$.

Now there is no $\tau\in (0,R)$ such that $m_\psi(\tau)=L$.
But $\psi$ extends to be continuous in $\cl(\D(0,R))$ and the family $(Y_t)_{t \in [0,R)}$ may be completed with a variable $Y_R$ which has $m_\psi(R)=M_\psi=L$, see Section \ref{section:extension of family}. If further $\sigma^2_\psi(R)=\V(Y_R)<\infty$, which occurs if $\sum_{n=0}^\infty n^2 b_n R^n<\infty$, then we may use formula \eqref{eq:formula maestra para grandes potencias Mf finito}
instead of formula \eqref{eq:formula maestra para grandes potencias}. By the same argument above, we then have

\begin{theo}\label{teor:large powers k/n to L=M_psi} If $M_\psi=L$ and $R<+\infty$ and $\sum_{n=0}^\infty n^2 b_n R^n <\infty$. If $$\lim_{n \to \infty} \frac{nL-k}{\sqrt{n}}=-\omega\,,$$ then
	$$\textsc{coeff}_{[k]} (\psi(z)^n) \sim Q_\psi\, \frac{e^{-\omega^2/(2\sigma_\psi^2(R))}}{\sqrt{2\pi}\,\sigma_\psi(R)} \frac{1}{\sqrt{n}}\frac{\psi^n(R)}{R^k}\,,$$
	where  $k$ is a multiple of $Q_\psi$.
\end{theo}

\begin{remark}
	There remains the case where $L=M_f$ and $R=+\infty$.  In this case,  since $M_f<
	~+\infty$, we have that $f$ is a polynomial of degree $\text{deg}(f)=L=M_f$ and thus, in particular, we have  that $L$ is an integer. In the extreme case when $k=Ln$, we have that $\textsc{coeff}_{[k]}(\psi^n)=b_L^n$, where $b_L$ is the $L$-th coefficient $\psi(z)$.
\end{remark}

\subsubsection{Binomial coefficients.}\label{section:binomial coeffs} As an illustration, we apply next Theorem \ref{teor:large powers k/n to L} to binomial coefficients. We take $\psi(z) = 1+z$,  which belongs to $\K$. In this case, we have
$$m_{\psi}(t) = \frac{t}{1+t}\,, \quad \mbox{and} \quad \sigma_\psi^2(t)=\frac{t}{(1+t)^2}\,, \quad \mbox{for $t \in (0,+\infty)$}\,.
$$
In particular,  $M_{\psi} = 1$.

Observe that
$$\textsc{coeff}_{[k]}(\psi(z)^n)=\binom{n}{k}\,, \quad \mbox{for $n \ge 1$ and $k\ge 1$}\,.$$

\medskip

\noindent $\bullet$ Let $p\in (0,1)$. For each $n \ge 1/p$, let $k=\lfloor pn\rfloor$. We have that
$$0\le p-\frac{k}{n}\le \frac{1}{n}\,.$$
Thus we may apply Theorem \ref{teor:large powers k/n to L} with $L=p$ and $\omega=0$, to deduce that
$$\binom{n}{\lfloor pn\rfloor}\sim \frac{1}{\sqrt{2\pi n p (1-p)}}\frac{1}{(1-p)^{n-k} p^k}\,, \quad \mbox{as $n \to \infty$}\,.$$
If $p=1/2$, we have that
$$\binom{n}{\lfloor n/2\rfloor}\sim\frac{2^{n{+}1}}{\sqrt{2\pi}\sqrt{n}}\,, \quad \mbox{as $n \to \infty$}\,.$$

\smallskip

\noindent $\bullet$ Let $p\in (0,1)$ and $\lambda \in \R$. For $n \ge N$, let $k=\lfloor p n+\lambda \sqrt{n}\rfloor$, where $N$ is chosen so that $ p n+\lambda \sqrt{n}\ge 1$, for $n \ge N$. Then
$$
\frac{k}{n}=p+\frac{\lambda}{\sqrt{n}}+O\Big(\frac{1}{n}\Big)\,, \quad \mbox{as $n \to \infty$}\,,
$$
and Theorem \ref{teor:large powers k/n to L} with $L=p$ and $\omega=\lambda$, gives us that
$$\binom{n}{\lfloor p n+\lambda \sqrt{n}\rfloor}\sim \frac{1}{\sqrt{2\pi n p (1-p)}}\frac{1}{(1-p)^{n-k} p^k}\, e^{-\lambda^2/(2p(1-p))}\,, \quad \mbox{as $n \to \infty$}\,.$$
For $p=1/2$ and $\lambda\in \R$ fixed, we have that
$$\binom{n}{\lfloor  n/2+\lambda \sqrt{n}\rfloor}\sim \frac{2^{n{+}1}}{\sqrt{2\pi}\sqrt{n}}\, e^{-2\lambda^2}\,, \quad \mbox{as $n \to \infty$}\,.$$

\subsubsection{Back to Local Central Limit Theorem} Let us assume that $\psi$ is the probability generating function of a certain random variable $X$ (with values in $\{0,1, 2\, \ldots\}$) which has mean $\mu $ and standard deviation $s$.

Assume further that $\psi$ has radius of convergence $R>1$ and that $Q_\psi=1$.

We have $M_\psi=\lim_{t \uparrow R} m_\psi(t)> m_\psi(1)=\mu$. For $\tau=1\in (0,R)$, we have
$m_\psi(\tau)=\mu$ and $\sigma_\psi(\tau)=s$.

Let $Y$ be the random variable $Y=(X-\mu)/s$. This variable $Y$ is a lattice random variable with gauge $h=1/s$, since $Q_\psi=1$. The variable $Y$ has $\E(Y)=0$ and $\E(Y^2)=1$.

\medskip

For each $n \ge 1$, let $S_n$ denote the random variable
$$
S_n=\frac{Y_1+\dots+Y_n}{\sqrt{n}}=\frac{X_1+\dots+X_n}{s\sqrt{n}}-\Big(\frac{\mu}{s}\Big)\sqrt{n}\,,$$
where $X_1, X_2, \ldots$ are independent copies of $X$ and $Y_j=(X_j-\mu)/s$, for $j \ge 1$.

For $k \in \Z$, denote $$v(k)=k\frac{1}{s\sqrt{n}}-\Big(\frac{\mu}{s}\Big) \sqrt{n}\,.$$
Let $\mathcal{V}_n$ denote the collection of values that  $S_n$ can take:
$$\mathcal{V}_n=\Big\{v(k): k \in \Z\Big\}\,.$$

\medskip

For $x \in \mathcal{V}_n$, let $k_x=s\sqrt{n} x +\mu n$, so that $v(k_x)=x$.
We have
$$\frac{k_x}{n}=\mu+ \frac{s}{\sqrt{n}} x\, .$$
Moreover, with the notations of the hypothesis of Theorem \ref{teor:large powers k/n to L}, we have that $L=\mu$ and $\omega=s x$.
Now we use that with $\tau=1$, we have that $m_\psi(1)=\mu(=L)$, $\psi(1)=1$, $\sigma_\psi(1)=s$ and $\omega^2/\sigma^2_\psi(1)=x^2$, and appealing to
Theorem \ref{teor:large powers k/n to L} we obtain that
$$\textsc{coeff}_{[k]}(\psi(z)^n)\sim \frac{1}{\sqrt{2\pi}}\frac{1}{s \sqrt{n}} e^{-x^2/2}\,, \quad \mbox{as $n \to \infty$}\,.$$

\medskip

Since
$$
\P(S_n=x)=\P(X_1+\dots+X_n=k_x)=\textsc{coeff}_{[k_x]}(\psi(z)^n)\,,$$
we deduce that
$$
\lim_{n \to \infty} s \sqrt{n}\, \P(S_n=x)=\frac{1}{\sqrt{2\pi}} e^{ -x^2/2}\,, \quad \mbox{as $n \to \infty$}\,,$$
which is \lq \lq consistent\rq\rq \, with the Local Central Limit Theorem as stated in Corollary \ref{teor:limite central local}.

\section{Index $k$ is little \lq o\rq\, of $n$}\label{section:k=o(n)}

In this section we deal with the large powers asymptotics when $k/n \to 0$ as $k,n \to \infty$.

For the discussion of this case and \textit{throughout this section we assume   that $\psi^\prime(0)\neq 0$}. This hypothesis implies that $Q_\psi=1$.

\

We will use formula \eqref{eq:formula maestra para grandes potencias} with an appropriate choice of $t$. To  precise the limit value of the integral term as $n \to \infty$, we shall appeal to the following lemma, akin to Lemma 2 in \cite{MeirMoon}.

\begin{lem}\label{lema:cota cuando a_1 no es cero} Suppose that $\psi \in \mathcal{K}$ and $\psi^{\prime}(0)\neq 0$, then for each $\theta_0 \in (0, \pi)$, there exists $c>0$ and $r>0$ such that
	$$
	|\E(e^{\imath\theta \breve{Y}_t})| = \frac{|\psi(te^{\imath \theta})|}{\psi(t)}\le e^{-ct}\,, \quad \mbox{if $t \le r$ and $\theta_0\le |\theta|\le \pi$}\,.$$
\end{lem}
\begin{proof} Without loss of generality we assume  that $b_0=1$. We have
	$$|\psi(te^{\imath \theta})|^2=1+2b_1 t \cos \theta +O(t^2)\,, \quad \text{ as } \quad t \downarrow 0,$$
	and
	$$|\psi(t)|^2=1+2b_1 t +O(t^2)\,, \quad \text{ as } \quad t \downarrow 0.$$
	Therefore,
	$$
	|\E(e^{\imath\theta \breve{Y}_t})|^2  = \frac{|\psi(te^{\imath \theta})|^2}{\psi(t)^2}=1+2b_1 t\, (\cos\theta -1)+O(t^2)\,, \quad \text{ as } \quad t \downarrow 0.$$
	
	For $\theta_0\le |\theta|\le \pi$, we have $\cos \theta\le 1-\delta$, for certain $\delta >0$, which may depend upon $\theta_0$. Since, by hypothesis, $b_1>0$, we then have, for certain $r \in (0,R)$, that
	$$
	|\E(e^{\imath\theta \breve{Y}_t})|^2  = \frac{|\psi(te^{\imath \theta})|^2}{\psi(t)^2}\le 1-2b_1 t\delta+O(t^2)\, , \quad \mbox{for all $t \leq r$}\,,$$
	Therefore, for $t \le r$ we have that
	$$
	|\E(e^{\imath\theta \breve{Y}_t})|^2  = \frac{|\psi(te^{\imath \theta})|^2}{\psi(t)^2}\le 1-b_1 t\delta\le e^{-b_1 \delta t}\,, $$as claimed.\end{proof}

For $n$ large enough, we have that $k/n<M_\psi$ and thus we may define uniquely $\tau_n\in (0,R)$ such that  $m_\psi(\tau_n)=k/n$. Observe that  $\tau_n\to 0$, as $n \to \infty$.

We will use here the  results from Section \ref{section:basic moments near 0} about the moments of the Khinchin family $(Y_t)_{t \in[0,R)}$ when $t$ is close to $0$.
Since $b_1>0$, we have that
$$
(\natural) \quad \lim_{t \downarrow 0} t^{1/2}\,  \E(|\breve{Y}_t|^3)=\sqrt{\frac{b_0}{b_1}}\,.
$$

As $m_\psi(t)= \dfrac{b_1}{b_0}t+O(t^2)$, we have that
$$\tau_n \sim \dfrac{b_0}{b_1}\frac{k}{n}\, , \quad \mbox{as $n \to \infty$}\, .$$
Besides,
$$\sigma_\psi^2(\tau_n)\sim \frac{k}{n}\, ,\quad \mbox{as $n \to \infty$}\,.$$

For $t=\tau_n$ the expression $(\natural)$ above tells us that
$$(\flat)\qquad \frac{\E(|\breve{Y}_{\tau_n}|^3)}{\sqrt{n}}\sim\sqrt{\frac{1}{{k}}}\, , \quad \mbox{as $k \to \infty$}\,.$$

\smallskip

If we choose $t=\tau_n$ in formula \eqref{eq:formula maestra para grandes potencias}, the integral term simplifies to
$$
I_n\triangleq\int\limits_{|\theta|\le \pi \sigma_\psi(\tau_n)  \sqrt{n}} \E\big(e^{\imath \theta\breve{Y}_{\tau_n} /\sqrt{n}}\big)^n d\theta\, .$$
Let us see that the integral $I_n$ converges to $\sqrt{2\pi}$ as $n \to \infty$.

\medskip

Theorem \ref{teor:convergencia uniforme TLC} and $(\flat)$ give that
$$\lim_{n \to \infty}\E\big(e^{\imath \breve{Y}_{\tau_n} \theta/\sqrt{n}}\big)^n=e^{-\theta^2/2}, \quad \mbox{for all $\theta \in \R$}\,.$$
Theorem \ref{teor:cota uniforme caracter TLC} and $(\flat)$ give that there exists $N \ge 1$ such that if $n \ge N$, then
$$\big|\E\big(e^{\imath \breve{Y}_{\tau_n} \theta/\sqrt{n}}\big)^n\big|\le e^{-\theta^2/3}\, , \quad \mbox{for all $|\theta|\le \sqrt{k}/5$}\,.$$
By dominated convergence we have
$$\lim_{n \to \infty} \int\limits_{|\theta|\le \sqrt{k}/5} \big|\E\big(e^{\imath \breve{Y}_{\tau_n} \theta/\sqrt{n}}\big)^n-e^{-\theta^2/2}\big|\, d\theta=0\, .$$

\

We apply lemma \ref{lema:cota cuando a_1 no es cero} with $\theta_0=1/10$ to ascertain that there exists $N\ge 1$ and $C>0$ such that for all $n\ge N$ we have that
$$|\E(e^{\imath\theta Y_{\tau_n}})|= \frac{|\psi(\tau_n e^{\imath \theta})|}{\psi(\tau_n)}\le e^{-c \tau_n}\, , \quad \mbox{for $\frac{1}{10}\le |\theta|\le \pi$}\,,$$
and, therefore, that
$$\begin{aligned}|\E(e^{\imath\theta Y_{\tau_n}/(\sigma_\psi(\tau_n)\sqrt{n})})|^{n} = \Big|\frac{\psi(\tau_n e^{\imath \theta/(\sigma_\psi(\tau_n)\sqrt{n})})}{\psi(\tau_n)}\Big|^n&\le e^{-c n \tau_n}
	\, , \\ &\quad \mbox{for $\frac{1}{10}\sigma_\psi(\tau_n)\sqrt{n}\le |\theta|\le \pi\sigma_\psi(\tau_n)\sqrt{n}$}\,. \end{aligned}$$

Since for $n$ large enough we have that $\frac{1}{10}\sigma_\psi(\tau_n)\sqrt{n}\le \sqrt{k}/5$, we may bound
$$\int\limits_{\sqrt{k}/5\le |\theta|\le  \pi\sigma_\psi(\tau_n)\sqrt{n}} \big|\E(e^{\imath \theta\breve{Y}_{\tau_n}/\sqrt{n}})^n-e^{-\theta^2/2}\big|d \theta\le e^{-c n\tau_n} (8\pi/3)\sqrt{k}+\int\limits_{|\theta|\ge \sqrt{k}/3} e^{-\theta^2/2}\, d\theta\,,$$
for $n$ large enough.

As $n \to \infty$, the first summand of this bound converges to $0$, since $n \tau_n\sim (a_0/a_1)k$  and $k \to \infty$, while the second summand converges to 0 since $k \to \infty$. We conclude  that
$$\lim_{n\to \infty}\int\limits_{\sqrt{k}\le |\theta|\le  \pi\sigma_\psi(\tau_n)\sqrt{n}} \big|\E(e^{\imath \theta\breve{Y}_{\tau_n}/\sqrt{n}})^n-e^{-\theta^2/2}\big|d \theta=0\, ,$$
and, therefore, that
\begin{equation}\label{eq:limite fuerte cuando k=o(n)}
	\lim_{n \to \infty} \int\limits_{|\theta|\le  \pi\sigma_\psi(\tau_n)\sqrt{n}} \big|\E(e^{\imath \theta\breve{Y}_{\tau_n}/\sqrt{n}})^n-e^{-\theta^2/2}\big|d \theta=0\,. \end{equation}
In particular
$$
\lim_{n \to \infty} I_n=\lim_{n \to \infty} \int\limits_{|\theta|\le  \pi\sigma_\psi(\tau_n)\sqrt{n}} \E(e^{\imath \theta\breve{Y}_{\tau_n}/\sqrt{n}})^n\, d\theta=\sqrt{2\pi}\,.$$

In sum,
\begin{theo} \label{teor:asymptotic powers with k=o(n)}If $\psi^\prime(0)\neq 0$, then
	$$\begin{aligned}
		\textsc{coeff}_{[k]}(\psi(z)^n)&\sim
		\frac{1}{\sqrt{2\pi}} \frac{\psi^n(\tau_n)}{\tau_n^k}\frac{1}{\sqrt{n}}\frac{1}{\sigma_\psi(\tau_n)}
		\\ & \sim \frac{1}{\sqrt{2\pi}} \frac{\psi^n(\tau_n)}{\tau_n^k} \frac{1}{\sqrt{k}}
		\, , \quad \mbox{as $n \to \infty$}\, ,\end{aligned}$$
	where $k/n\to 0$ and $\tau_n$ is given by $m_\psi(\tau_n)=k/n$ (for $n$ large enough).
\end{theo}

\subsection{With further information on $k/n$}\label{section:with more information on k/n}

If  further information on how fast $k/n$ tends to 0 is available, then we may express the asymptotic formula \eqref{teor:asymptotic powers with k=o(n)} directly in terms of $k$ and $n$ and not on $\tau_n$.

We maintain the assumption that $\psi^\prime(0)>0$.

\

The function $m_\psi$ is holomorphic and injective near 0. Its inverse $m_\psi^{-1}$ is actually the solution of Lagrange's equation  with data $z/m_\psi(z)=\psi(z)/\psi^{\prime}(z)$. Besides, $\ln \psi $ is holomorphic near 0, since as $\psi \in \K$, we have $\psi(0)>0$. Thus Lagrange's inversion formula gives the expansion
$$
\ln \psi \big(m_\psi^{-1}\big)(z)=\ln b_0+\sum_{j=1}^\infty B_j z^j,
$$
where
\begin{equation}\label{eq:formula of Bs}B_j=\frac{1}{j} \textsc{coeff}_{[j-1]}\bigg(\Big(\frac{\psi^\prime}{\psi}\Big) \Big(\frac{\psi}{\psi^\prime}\Big)^j\bigg)=\frac{1}{j} \textsc{coeff}_{[j-1]}\bigg( \Big(\frac{\psi}{\psi^\prime}\Big)^{j-1}\bigg)\, , \quad \mbox{for $j \ge 1$}\,.\end{equation}
For each $j \ge 1$, the coefficient $B_j$ depends only on $b_0, b_1, \ldots, b_j$.

\smallskip

Likewise,
$$
\ln \frac{m_\psi^{-1}(z)}{z} =\ln \Big(\frac{\psi}{\psi^\prime}\Big) \big(m_\psi^{-1}\big)(z)=\ln \frac{b_0}{b_1}+\sum_{j=1}^\infty C_j z^j,
$$
where
\begin{equation}\label{eq:formula of Cs}
	C_j=\frac{1}{j}\textsc{coeff}_{[j]} \Big(\frac{\psi}{\psi^\prime}\Big)^j=\frac{j{+}1}{j} B_{j{+}1}\, , \quad \mbox{for $j \ge 1$}\,.\end{equation}

\

Now, with the notations of the general discussion of the case $k/n \to 0$, we have that
$$
\ln \psi(\tau_n)^n =n \ln \psi(\tau_n)=n \ln \psi\Big(m_\psi^{-1}\Big(\frac{k}{n}\Big)\Big)=n \ln b_0+\sum_{j=1}^\infty B_j (k^j/n^{j{-}1})\,,$$
and also that,
$$
\ln \tau_n^k=k \ln \tau_n=k \ln m_\psi^{-1}\Big(\frac{k}{n}\Big)=k \ln \frac{b_0}{b_1}+k \ln \frac{k}{n}+ \sum_{j=1}^\infty C_j \frac{k^{j{+}1}}{n^j}.$$

Using formula \eqref{eq:formula of Cs}, and that $B_1=1$, we may write
\begin{equation}
	\begin{aligned}
		\ln \psi(\tau_n)^n -\ln \tau_n^k=(n-k) \ln b_0+k \ln b_1+k\ln\frac{n}{k}+k-\sum_{j=2}^\infty \frac{B_j}{j{-}1}\frac{k^j}{n^{j{-}1}}.
	\end{aligned}
\end{equation}

Let $\lambda=\limsup_{n \to \infty} {\ln k}/{\ln n}$ and let $J$ the smallest integer such that
$J> \lambda/(1-\lambda)$; if $\lambda=1$, we set $J=+\infty$. Thus
$$\sum_{j=J+1}^\infty \frac{B_j}{j{-}1}\frac{k^j}{n^{j{-}1}}=n O\Big(\frac{k}{n}\Big)^{J{+}1}=o(1)\,, \quad \mbox{as $n \to \infty$}\,,
$$
and then
$$\frac{\psi(\tau_n)^n}{\tau_n^k}\sim b_0^{n{-}k} b_1^k \frac{n^k \, e^k}{k^k} \exp\Big(-\sum_{j=2}^J \frac{B_j}{j{-}1}\frac{k^j}{n^{j{-}1}}\Big)\,, \quad \mbox{as $n \to \infty$}\,.
$$
Therefore
\begin{theo} \label{teor:asymptotic powers with k=o(n) bis} If $\psi^\prime(0)\neq 0$, and $\lambda$ and $J$ as above, then
	$$
	\textsc{coeff}_{[k]}(\psi(z)^n)
	\sim \frac{1}{\sqrt{2\pi}}  b_0^{n{-}k} b_1^k \frac{n^k \, e^k}{k^k \sqrt{k}} \exp\Big(-\sum_{j=2}^J \frac{B_j}{j{-}1}\frac{k^j}{n^{j{-}1}}\Big)
	\, , \quad \mbox{as $n \to \infty$}\, .$$
\end{theo}

If $\lambda <1/2$, then $J=1$, and then
\begin{equation}\label{eq:k/n to infty and lambda less 1/2}
	\textsc{coeff}_{[k]}(\psi(z)^n)
	\sim  b_0^{n{-}k} b_1^k \frac{n^k }{k!}\sim \textsc{coeff}_{[k]}((b_0+b_1z)^n)
	\, , \quad \mbox{as $n \to \infty$}\, .\end{equation}

For the particular case where, $k=\lfloor \sqrt{n}\rfloor $, (with $\lambda=1/2$, and $J=2$), we have that
$$
\textsc{coeff}_{[k]}(\psi(z)^n)
\sim  b_0^{n{-}k} b_1^k \frac{n^k }{k!} e^{-B_2}
\, , \quad \mbox{as $n \to \infty$}\, .$$
where $B_2=(1/2)-(b_2 b_0)/b_1^2$.

\subsection{Particular case: $k$ fixed, while $n \to \infty$}

In this particular case, where $k$ remains fixed, while $n \to \infty$,  the analysis of the asymptotic behaviour of $\textsc{coeff}_{[k]}(\psi(z)^n)$ just requires the multinomial theorem.

\medskip

Observe that
$$\textsc{coeff}_{[k]}(\psi(z)^n)=\textsc{coeff}_{[k]}\Big(\sum_{j=0}^k b_j z^j\Big)^n\,.$$
The multinomial theorem gives then that
$$\textsc{coeff}_{[k]}(\psi(z)^n)=\sum \frac{n!}{j_0!\dots j_k!} b_0^{j_0} \dots b_k^{j_k}\, ,$$
where the sum extends to all $k{+}1$-tuples $(j_0,\ldots, j_k)$ of nonnegative integers  such that
$$\begin{aligned}j_0+j_1+\dots+j_k&=n,\\
	j_1+2j_2+\dots+k j_k&=k.\end{aligned}$$
For such a tuple, we have $n-k\le j_0\le n$ and $j_i\le k$, for $1\le i \le k$.
We write $j_0=n-l$, with $0\le l\le k$ and classify by $l$:
$$\textsc{coeff}_{[k]}(\psi(z)^n)=\sum_{l=0}^k \binom{n}{l}b_0^{n{-}l}\, \underset{\triangleq C_l}{\underbrace{\sum \begin{pmatrix}j_1+\dots+j_k\\j_1,\dots, j_k
		\end{pmatrix} b_1^{j_1}\dots b_k^{j_k}}}\,.$$
For each integer $l, 0 \le l \le k$, the interior sum $C_l$ extends to all $k$-tuples with $j_1+\dots+j_k=l$ and $j_1+2j_2+\dots+kj_k=k$. Observe that $C_l$ does not depend upon $n$.
Therefore
$$ \textsc{coeff}_{[k]}(\psi(z)^n)=\sum_{l=0}^k \binom{n}{l}b_0^{n{-}l} C_l,$$
is a polynomial in $n$ whose degree $\gamma$ is given by the largest $l\le k$ such that $C_l \neq 0$, and
$$ \textsc{coeff}_{[k]}(\psi(z)^n)\sim \frac{1}{\gamma!} b_0^{n{-}\gamma} C_\gamma\, n^\gamma\, , \quad \mbox{as $n \to \infty$}\,.$$

\

\noindent $\bullet$  \quad If $b_1\neq 0$,  then  $C_k\neq 0$, and $\gamma=k$. In fact, the sum $C_k$ contains just one summand corresponding to $j_1=k$ and $j_2=\ldots=j_k=0$, and thus $C_k=b_1^k$. We have in this case that
$$\textsc{coeff}_{[k]}(\psi(z)^n)\sim \frac{1}{k!} b_0^{n{-}k} b_1^k\, n^k\sim \binom{n}{k}b_0^{n{-}k} b_1^k\, , \quad \mbox{as $n \to \infty$}\,.$$
In other terms and asymptotically speaking, the  $k$-th coefficient of $\psi(z)^n$ behaves as the $k$-th coefficient of $(b_0+b_1z)^n$, as $n \to \infty$. Compare with the formula \eqref{eq:k/n to infty and lambda less 1/2}, where $k \to \infty$ with $n$, but slowly.

\

\noindent $\bullet$  \quad If $b_1=0$, then the nonzero summands of $C_l$ must have $j_1=0$ and they must satisfy $j_2+\dots+j_k=l$ and $2j_2+\dots+k j_k=k$. Thus, if $C_l$ is not zero, then $2l \le k$, and so, $\gamma \le k/2$.

Assume further that $b_2 \neq 0$. We distinguish now between  $k$ even and $k$ odd.

If $k=2q$, then  $C_q$ has only one nonzero summand, the one corresponding to $j_2=q$. Thus $C_q=b_2^q$,  $\gamma=q=k/2$, and
$$\textsc{coeff}_{[k]}(\psi(z)^n)\sim \frac{1}{(k/2)!} b_0^{n{-}k/2} b_2^{k/2}\, n^{k/2}\, , \quad \mbox{as $n \to \infty$}\,.$$

If $k=2q+1$, then $\gamma \le q$. The only nonzero summand of $C_q$ has $j_2=q-1, j_3=1$. Thus $C_q=q b_2^{q-1}b_3$.

If $b_3 \neq 0$, then $C_q\neq 0$, $\gamma=q$ and
$$\textsc{coeff}_{[k]}(\psi(z)^n)\sim \frac{1}{((k{-}3)/2){\mbox{\large{!}}}}\,  b_0^{n-(k-1)/2} b_2^{(k{-}3)/2} b_3 \, n^{(k{-}1)/2}\, , \quad \mbox{as $n \to \infty$}\,.$$

If $b_3=0$, then $C_q=0$ and actually the degree $\gamma$ is at most $k/3$.

\

\noindent  $\bullet$ \quad In general, if $b_1=\ldots=b_{m{-}1}=0$, and $b_m\neq 0$, then for each $h, 0 \le h \le m{-}1$, the coefficients of index $k\equiv h \mod(m)$, satisfy an asymptotic formula
$$  \textsc{coeff}_{[k]}(\psi(z)^n)\sim \Omega_{k,h} n^{\gamma_{k,h}}\, , \quad \mbox{as $n \to \infty$}\,,$$
where $\gamma_{k,h}\le (k-h)/m$ and $\Omega_{k,h}>0$.

For $h=0$, i.e., for those $k$ which are multiples of $m$, we have $\gamma_{k,0}=k/m$
and, actually, that
$$  \textsc{coeff}_{[k]}(\psi(z)^n)\sim \frac{1}{(k/m)!} b_0^{n{-}k/m} b_m^{k/m}\, n^{k/m}\, , \quad \mbox{as $n \to \infty$}\,.$$

In particular, if $Q=Q_\psi >1$ and $k$ is a multiple of $Q$, then
$$ \textsc{coeff}_{[k]}(\psi(z)^n)\sim \frac{1}{(k/Q)!} b_0^{n-k/Q} b_Q^{k/Q} \, n^{k/Q}\, , \quad \mbox{as $n \to \infty$}\,.$$
Observe the $b_Q>0$.
If $Q$ is not a divisor of $k$, then, of course,  $\textsc{coeff}_{[k]}(\psi(z)^n)=0$. for every $n \ge 1$.

\section{Power $n$ is little \lq o\rq\, of index $k$}\label{section:n=o(k)}

To obtain a large powers asymptotic formula in this case where $k/n \to \infty$, as $n \to \infty$, we require more from $\psi$, namely, we will asume that $\psi$ is uniformly Gaussian. In an appendix of this paper, we describe at some length these so called uniformly Gaussian power series.

Since uniformly Gaussian power series are strongly Gaussian, Hayman's asymptotic formula \eqref{eq:asymptotic formula coeffs} is valid for $\psi$ and thus $Q_\psi=1$.

The assumption of uniform Gaussianity  on $\psi$ is to be compared with the assumptions  of Gardy in \cite[Theorems 5 and 6]{Gardy}, see also \cite[Section 6]{Gardy}.

The exponential generating function of set partitions, $e^{e^z-1}$, and the ordinary generating function of partitions of integers, $\prod_{j=1}^\infty 1/(1-z^j)$ are examples of uniformly Gaussian power series.

\

Since $\psi$ is uniformly Gaussian, we have that $M_\psi=+\infty$, and thus we may define (uniquely) $\tau_n\in (0,R)$ as being such that $m_\psi(\tau_n)=k/n$. Observe that $\tau_n \uparrow R$, as $n \to \infty$.

Now, formula \eqref{eq:formula maestra para grandes potencias} of Lemma \ref{lemma:formula maestra para grandes potencias}, with $t=\tau_n$, gives us that
$$
\textsc{coeff}_{[k]}(\psi(z)^n)=
\frac{1}{2\pi} \frac{\psi^n(\tau_n)}{\tau_n^k}\frac{1}{\sqrt{n}}\frac{1}{\sigma_\psi(\tau_n)}\int\limits_{|\theta|\le \pi \sigma_\psi(\tau_n)  \sqrt{n}} \E\big(e^{\imath \theta\breve{Y}_{\tau_n}/\sqrt{n}}\big)^n d\theta\, .
$$
Since $\tau_n \to R$, uniform Gaussianity of $\psi$ implies that
$$
\lim_{n \to \infty} \int\limits_{|\theta|\le \pi \sigma_\psi(\tau_n)  \sqrt{n}} \E\big(e^{\imath \theta\breve{Y}_{\tau_n}/\sqrt{n}}\big)^n d\theta=\sqrt{2\pi}\,,$$
and we conclude that

\begin{theo} If $\psi$ is uniformly Gaussian, then
	$$
	\textsc{coeff}_{[k]}(\psi(z)^n)\sim
	\frac{1}{\sqrt{2\pi}} \frac{\psi^n(\tau_n)}{\tau_n^k}\frac{1}{\sqrt{n}}\frac{1}{\sigma_\psi(\tau_n)}\, , \quad \mbox{as $n \to \infty$}\,,$$
	where $k/n\to \infty$ as $n \to \infty$ and $\tau_n$ is such that $m_\psi(\tau_n)=k/n$.
\end{theo}

\section{Coefficients of $h(z)\psi(z)^n$}\label{section:coef h psi^n}

Next we consider  asymptotics of the coefficients of $h(z)\psi(z)^n$, as $n \to \infty$, where both $h$ and $\psi$ are in $\K$. This is to be used in Section \ref{section:prob gen functions} which deals  with Lagrangian distributions and Galton-Watson processes: the power series $h$ would codify the initial offspring distribution while the power series $\psi$ would codify the probability distributions subsequent offspring distributions.

\smallskip

We assume here from the start  that $Q_\psi=1$.

Let $h(z) = \sum_{j=0}^{\infty}c_j z^j$ be a power series in $\K$. We assume throughout that the radius of convergence of $h$ is at least $R_\psi$, the radius of convergence of $\psi$.

We treat jointly the cases in which $k\asymp n$ and $k=o(n)$, with $k \to \infty$, by adjusting  the  arguments in the previous sections (whose notations we will use liberally) in which $h \equiv 1$.

Assume the hypothesis of Theorem \ref{teor:asymptotic powers with k comp n}, when $k \asymp n$, and of Theorem \ref{teor:asymptotic powers with k=o(n)}, when $k=o(n)$ and $k \to \infty$.

The $k$-th coefficient of $h(z)\psi(z)^n$ is given, see \eqref{eq:formula maestra para grandes potencias funcion h}, by
$$ \frac{1}{2\pi} \frac{\psi^n(\tau_n)}{\tau_n^k}\frac{1}{\sqrt{n}}\frac{1}{\sigma_\psi(\tau_n)}\int\limits_{|\theta|\le \pi \sigma_\psi(\tau_n)  \sqrt{n}} h(\tau_n e^{\imath \theta /(\sigma_\psi(\tau_n)\sqrt{n})})\E\big(e^{\imath \theta\breve{Y}_{\tau_n}/\sqrt{n}}\big)^n d\theta\, ,$$
where we are taking $\tau_n$ such that $m_\psi(\tau_n) = k/n$.

Denote by $I_n$ the integral on the right hand-side of the previous expression.

\smallskip

For $0 \le s\le r<R$ and $\phi \in \R$ we have
$$\big|h(se^{\imath \phi})-h(s)\big|\le \max\limits_{|z|\le r}|h^\prime(z)|\, s\,  |\phi|\, .$$
For $k \asymp n$, we have $\tau_n\le m_\psi^{-1}(B)<R$, and for $k=o(n)$ we have $\tau_n \to 0$, so, in both cases, there exists a constant $K>0$ such that
$$
(\nabla) \qquad\Big|h(\tau_n e^{\imath \theta/(\sigma_\psi(\tau_n)\sqrt{n})})-h(\tau_n)\Big| \le K \frac{|\theta|}{\sigma_\psi(\tau_n) \sqrt{n}}\, .$$
Denote by $\mathcal{D}_n$ the interval $\{\theta \in \R: |\theta|\le \pi \sigma_\psi(\tau_n) \sqrt{n}\}$.

Using the bound $(\nabla)$ and writing
$$\E\big(e^{\imath \theta\breve{Y}_t/\sqrt{n}}\big)^n=\Big(\E\big(e^{\imath \theta\breve{Y}_t/\sqrt{n}}\big)^n-e^{-\theta^2/2}\Big)+e^{-\theta^2/2}\, ,$$
we conclude that
$$\begin{aligned}I_n&=h(\tau_n)\int\limits_{\mathcal{D}_n} e^{-\theta^2/2}d\theta
	+h(\tau_n)\int\limits_{\mathcal{D}_n} \Big(\E\big(e^{\imath \theta\breve{Y}_{\tau_n}/\sqrt{n}}\big)^n-e^{-\theta^2/2}\Big)d\theta\\
	&+O\Big(\int\limits_{\mathcal{D}_n} \Big|\E\big(e^{\imath \theta\breve{Y}_{\tau_n}/\sqrt{n}}\big)^n-e^{-\theta^2/2}\Big| \frac{|\theta|}{\sigma_\psi(\tau_n) \sqrt{n}}d\theta\Big)
	+O\Big(\int\limits_{\mathcal{D}_n} e^{-\theta^2/2} \frac{|\theta|}{\sigma_\psi(\tau_n) \sqrt{n}}d\theta\Big)\,.
\end{aligned}
$$

For  $\theta\in\mathcal{D}_n$, we have that $\dfrac{|\theta|}{\sigma_\psi(\tau_n)\sqrt{n}}\le \pi$, and thus we see that the third term in the sum above tends to 0, by virtue of \eqref{eq:limite fuerte cuando k como n} (if $k\asymp n$) or \eqref{eq:limite fuerte cuando k=o(n)} (if $k=o(n)$).

The second term in the sum tends to 0, since $h(\tau_n)$ is bounded and the integral converges to 0, by  \eqref{eq:limite fuerte cuando k como n} and  \eqref{eq:limite fuerte cuando k=o(n)}, respectively.

The fourth term in the sum tends to 0 since $\sigma_\psi(\tau_n)\sqrt{n}\asymp \sqrt{k}$ and $\int\limits_\R e^{-\theta^2/2} |\theta| d\theta =2$.

The first term in the sum is $h(\tau_n)\sqrt{2\pi}+o(1)$, since $h(\tau_n)$ is bounded, and so
$$I_n=h(\tau_n)\sqrt{2\pi}+o(1)\, , \quad \mbox{as $n \to \infty$}\,,$$
but, using that $h(\tau_n)$ is bounded from below by $h(0)$ and recalling that $h \in \mathcal{K}$, we conclude that
$$I_n=h(\tau_n)\sqrt{2\pi}(1+o(1))\, , \quad \mbox{as $n \to \infty$}\,.$$

We summarize the discussion above as follows.

\begin{theor} With the notations above,
\begin{itemize}\item
For the case $k \asymp n$, one has that
$$
\textsc{coeff}_{[k]}(h(z)\psi(z)^n)\sim
\frac{1}{\sqrt{2\pi}} \frac{h(\tau_n) \psi^n(\tau_n)}{\tau_n^k}\frac{1}{\sqrt{n}}\frac{1}{\sigma_\psi(\tau_n)}\, , \quad \mbox{as $n \to \infty$}\, .$$

\item For the case $k=o(n)$ and $k \to \infty$, one has  that
$$
\textsc{coeff}_{[k]}(h(z)\psi(z)^n)
\sim \frac{h(0)}{\sqrt{2\pi}} \frac{\psi^n(\tau_n)}{\tau_n^k} \frac{1}{\sqrt{k}}
\, , \quad \mbox{as $n \to \infty$}\, .$$
\end{itemize}
\end{theor}
Observe that, in the case $k=o(n)$ and $k \to \infty$, we have that $h(\tau_n)\to h(0)$ as $n \to \infty$.

\

If more information is available on how $k/n$ tends to 0, as discussed in Section \ref{section:with more information on k/n}, we could apply the argument and the conclusions obtained there to refine the above asymptotic formula.

\medskip

Suppose that $k/n \rightarrow L > 0$, as $n \rightarrow \infty$ and that equation \eqref{eq:lim L and omega} is satisfied. Assume that $L < M_{\psi}$. Then for $\tau = m_\psi^{-1}(L)$, we deduce, using the Hayman's formula \eqref{eq:formula maestra para grandes potencias funcion h} and arguing as in the case $h \equiv 1$, that
\begin{equation}
	\label{eq:formula_limit}
	\textsc{coeff}_{[k]}(h(z)\psi(z)^n)\sim
	\frac{1}{\sqrt{2\pi}} \frac{h(\tau) \psi^n(\tau)}{\tau^k}\frac{1}{\sqrt{n}}\frac{1}{\sigma_\psi(\tau)}\, , \quad \mbox{as $n \to \infty$}\, .
\end{equation}

Recall that the power series $h$ has radius of convergence $R_\psi$ and that $\tau \in (0,R_\psi)$.

\

To conclude, we consider the case in which $k$ is fixed and $n\to \infty$. Denote $Q=Q_{\psi}= \{n \geq 1 : b_n >0\}$.
Then
$$\begin{aligned}\textsc{coeff}_{[k]}(h(z)\psi(z)^n)&=\sum_{j=0}^k c_j \textsc{coeff}_{[k]}(z^j\psi(z)^n)\\
	&=\sum_{j=0}^k c_j \textsc{coeff}_{[k{-}j]}(\psi(z)^n).\\
\end{aligned}$$
The only $(k{-}j)$-th coefficients of $\psi(z)^n$ which are nonzero are the coefficients with indices $j \equiv k, \mod Q$. Appealing to the case $h\equiv 1$,
we have that
$$\begin{aligned}
	\textsc{coeff}_{[k]}(h(z)\psi(z)^n) & = \sum_{\substack{0 \leq j \leq k \\ j \equiv k \,, \text{mod} Q}}c_{j} \textsc{coeff}_{[k{-}j]}(\psi(z)^n)\\
	&= \sum_{\substack{0 \leq j \leq k \\ j \equiv k \,, \text{mod} Q}} c_{j} \frac{1}{((k{{-}}j)/Q)!} b_0^{n{-}(k{-}j)/Q}\, b_Q^{(k{-}j)/Q}\, n^{(k{-}j)/Q}\, (1+o(1))\,.
\end{aligned}$$

Denoting $j_0 = \displaystyle\min\{0 \leq j \leq k : j \equiv k, \mod Q \text{ and } c_j \neq 0\}$, we have that
$$\begin{aligned}
	\textsc{coeff}_{[k]}(h(z)\psi(z)^n) & \sim c_{j_0} \frac{1}{((k{-}j_0)/Q)!} b_0^{n{-}(k{-}j_0)/Q}\, b_Q^{(k{-}j_0)/Q}\, n^{(k{-}j_0)/Q},  \quad \mbox{as $n \to \infty$}\, . \\
\end{aligned}$$

In  particular, if $h(0)=c_0\neq0$ and if  $Q$ is a divisor of $k$,  then
$$\begin{aligned}
	\textsc{coeff}_{[k]}(h(z)\psi(z)^n) & \sim c_{0} \frac{1}{(k/Q)!} \, b_0^{n{-}k/Q}\, b_Q^{k/Q}\, n^{k/Q}, \quad \mbox{as $n \to \infty$}\, . \\
\end{aligned}$$
Observe that in this case, $j_0=0$.

\section{Coefficients of solutions of Lagrange's equation}\label{section:meir-moon}

Next we shall apply the asymptotic results about large powers of Section \ref{section:large powers} to obtain asymptotic formulae for the coefficients of solutions of Lagrange's equation when the data $\psi$ of the equation belongs to $\mathcal{K}$.

\subsection{Lagrange's equation}

We start with a function $\psi(z)$ which is holomorphic in a neighborhood of $z = 0$, has  radius of convergence $R > 0$ and is such that $\psi(0) \neq 0$.
Consider \textit{Lagrange's equation}:
$$(\dag)\qquad g(w)=w \psi(g(w))\, .$$

The \textit{solution} $g(w)$ of Lagrange's equation with \textit{data} $\psi(z)$ is a holomorphic function $g(w)=\sum_{n=1}^\infty A_n w^n$ which satisfies
$(\dag)$ in a neighborhood of $w=0$.

Assume that the data $\psi(z)$ of Lagrange's equation has power series expansion $$\psi(z)=\sum_{n=0}^\infty b_n z^n, \quad \mbox{for all $z \in \D(0,R)$}.$$

The holomorphic function $g(w)$, solution of Lagrange's equation with data $\psi(z)$, is unique. In fact, the coefficients $A_n$ of the Taylor expansion of $g(w)$ around $w = 0$ are given by \textit{Lagrange's inversion formula}:
$$A_n=\frac{1}{n}\textsc{coeff}_{[n{-}1]}(\psi(z)^n)\, , \quad \mbox{for all $n \ge 1$}\, .$$
For $n=0$, we have $A_0=g(0)=0$.

This formula is exact for each coefficient $A_n$ of $g$ such that $n \geq 1$.

In a more general setting, consider the coefficients of $H(g(z))$, where $g$ is the solution of Lagrange's equation with data $\psi$ and $H$ is a holomorphic function with nonnegative coefficients around $z = 0$, then the \textit{extended Lagrange's inversion formula} gives
\begin{align} \label{eq: Lagrange_general}
	\textsc{coeff}_{[n]}(H(g(z))) = \frac{1}{n}\textsc{coeff}_{[n{-}1]}(H^{\prime}(z)\psi(z)^n), \quad \mbox{for all $n \ge 1$}\, ,
\end{align}
and $\textsc{coeff}_{[0]}(H(g(z))) = H(0)$.

A particular instance of this formula, that will be of interest later on,  is given by the choice $H(z) = z^q$ for an integer $q \geq 1$. In this case we have
\begin{align}
	\textsc{coeff}_{[n]}(g(z)^q) = \frac{q}{n}\textsc{coeff}_{[n{-}q]}(\psi(z)^n), \quad \mbox{for all $n \ge 1$}\, ,
\end{align}
and $\textsc{coeff}_{[0]}(g(z)^q) = H(0) = 0.$

\smallskip

If $Q_\psi>1$, then the only nonzero coefficients $b_n$ are those where the index $n$ is a multiple of $Q_\psi$.Therefore, for the solution $g(w)$ of Lagrange's equation with data $\psi$, the only nonzero coefficients $A_n$ are those where the index $n$ verifies that $n-1$ is a multiple of $Q_{\psi}$.

\smallskip

If $\psi(0)=0$, then the only solution of Lagrange's equation is $g \equiv 0$, but, recall, nonetheless that  $\psi \in \K$ implies that $\psi(0) >0$.

\subsection{The Otter-Meir-Moon theorem and some extensions}

The Otter-Meir-Moon Theorem, Theorem \ref{teor:Otter-Meir-Moon} below,  comes from  \cite[Theorem 4]{Otter} and \cite[Theorem 3.1]{MeirMoonOld}. It  gives an asymptotic formula for the coefficients $A_n$ of $g$, when $n{-}1$ is a multiple of $Q_\psi$ and $n\to \infty$, using minimal (but crucial) information about the function $\psi$.

We distinguish three cases, according as  $M_\psi$ is $>1, =1$ or $<1$.

\noindent $\blacklozenge$   $M_\psi>1$. This is the original assumption of both Otter and Meir-Moon; it means that there is a unique $\tau\in (0,R)$ such that $m_\psi(\tau)=1$.

\begin{theorem}[Otter-Meir-Moon theorem]\label{teor:Otter-Meir-Moon} Let $\psi(z)$ be a power series in $\K$ with radius of convergence $R > 0$ and let $\psi(z)=\sum_{n=0}^\infty b_n z^n$ be its power series expansion.
	
	Assume that $M_\psi >1$ and let  $\tau \in (0,R)$ be given by $m_\psi(\tau)=1$.
	
	Then the coefficients $A_n$ of the solution $g(w)$ of Lagrange's equation verify that
	\begin{itemize}
		\item if $n\not\equiv 1 \mod{Q_\psi}$, then  $A_n=0$,
		\item for the indices $n$ such that $n\equiv 1 \mod{Q_\psi}$ we have the asymptotic formula
		\begin{equation}\label{eq:asintotico meir-moon}A_n \sim \frac{Q_\psi}{\sqrt{2\pi}}\, \frac{\tau}{\sigma_\psi(\tau)}  \frac{1}{n^{3/2}} \, \Big(\frac{\psi(\tau)}{\tau}\Big)^n\, , \quad \mbox{as $n \to \infty$}\,.\end{equation}
	\end{itemize}
\end{theorem}

\

Theorem \ref{teor:Otter-Meir-Moon} of Otter and Meir-Moon follows readily from Theorem \ref{teor:large powers k/n to L} with $L=1$ and $\omega=0$, since $k=n-1$.

Given that
$$\sigma^2_\psi(t)=m_\psi(t)(1-m_\psi(t))+t^2\frac{\psi^{\prime\prime}(t)}{\psi(t)}\, , \quad \mbox{for all $t \in[0,R)$}\, ,$$ we have
\begin{align}\label{eq: varianza_tau}
	\frac{\psi^{\prime\prime}(\tau)}{\psi(\tau)}=\frac{\sigma^2_\psi(\tau)}{\tau^2}\, .
\end{align}
We thus may rewrite the conclusion of Theorem \ref{teor:Otter-Meir-Moon}  (in Laplace's method or saddle point approximation style) as
$$
A_n\sim Q_\psi\, \sqrt{\frac{\psi(\tau)}{2\pi \psi^{\prime\prime}(\tau)}}\,  \frac{1}{n^{3/2}} \Big(\frac{\psi(\tau)}{\tau}\Big)^n\, , \quad \mbox{when $n{-}1$ is a multiple of $Q_\psi $ and $n \to \infty$}\,.
$$

\

\noindent $\blacklozenge$  $M_\psi=1$ (and $R<+\infty$). For the case $M_\psi=1$, there is an Otter-Meir-Moon like theorem under  a certain non-degeneracy condition on $\psi$. Recall the notations of Section \ref{section:range of mean}.

\begin{theor}\label{teor:Meir-Moon Mpsi=1} Let $\psi(z)$ be a power series in $\K$ with radius of convergence $R > 0$ and let $\psi(z)=\sum_{n=0}^\infty b_n z^n$ be its power series expansion.
	
	Assume  that $M_\psi=1$ and that $R<\infty$, and besides that $\sum_{n=0}^\infty n^2 b_n R^n <+\infty$ or equivalently $\psi^{\prime\prime}(R)<+\infty$.

	Then the coefficients $A_n$ of the solution $g(w)$ of Lagrange's equation verify that
	\begin{itemize}
		\item if $n\not\equiv 1 \mod{Q_\psi}$, then  $A_n=0$,
		\item for the indices $n$ such that $n\equiv 1 \mod{Q_\psi}$ we have the asymptotic formula
		\begin{equation}\label{eq:asintotico meir-moon tau=R}A_n \sim \frac{Q_\psi}{\sqrt{2\pi}}\, \frac{R}{\sigma_\psi(R)}  \frac{1}{n^{3/2}} \, \Big(\frac{\psi(R)}{R}\Big)^n\, , \quad \mbox{as $n \to \infty$}\,.\end{equation}
	\end{itemize}
\end{theor}

The power series $\psi$ extends to be continuous in $\cl(\D(0,R))$ and the Khinchin family $(Y_t)_{t \in [0,R)}$ extends (continuously in distribution) to include a variable $Y_R$ with mean 1 and variance $\sigma_{\psi}^2(R)=\lim_{t \uparrow R} \sigma_{\psi}^2(t)$ which is finite because by hypothesis   $\sum_{n=0}^\infty n^2 b_n R^n <+\infty$. See Section \ref{section:range of mean}.

Theorem \ref{teor:Meir-Moon Mpsi=1} follows readily from Theorem \ref{teor:large powers k/n to L=M_psi} with $L=1$ and $\omega=0$, since $k=n-1$.

This limit case of the Otter-Meir-Moon Theorem appears as Remark 3.7 in \cite[Chapter 3]{Drmota}; the proof suggested there, of a different nature than  the one above,   comes from the Appendix of \cite{Janson}. See also \cite[Section 2.3.1]{Minami}. In there you may find  the same asymptotic result under the assumption that $\psi$ is a probability generating function satisfying  $\psi^{(4)}(R)=\lim_{t\uparrow R} \psi^{(4)}(t)<\infty$.

\

Observe that, in particular, it follows from Theorem \ref{teor:Otter-Meir-Moon} that when $M_\psi>1$ and with $m_\psi(\tau)=1$ the radius of convergence of the solution $g$ is $\tau/\psi(\tau)$. When $M_\psi=1$ (and $R<+\infty$), the radius of convergence of $g$ is $R/\psi(R)$.

\

If $M_\psi=1$ and $R=+\infty$, then $\psi$ is a polynomial of degree 1. Thus $\psi(z)=a+bz$, with both $a,b>0$. In this case the solution $g(z)$ of Lagrange equation with data $\psi$ is simply
$$g(z)=\frac{az}{1-bz}=\sum_{n=1}^\infty (ab^{n{-}1}) z^n\,.$$
and thus
$$A_n=ab^{n{-}1}\,, \quad \mbox{para $n \ge 1$}\,.$$

\

\noindent $\blacklozenge$   $M_\psi<1$. In this case, we necessarily have that $R<\infty$, since $R=\infty$ and $M_\psi<1$ would imply that $\psi$ is a constant, which is not allowed. See Lemma \ref{lemma:char of Mf finite}.

We assume as above that $\sum_{n=0}^\infty n^2 b_n R^n <\infty$. As in the preceding  case $M_\psi=1$ above, the power series $\psi$ is continuous in $\cl(\D(0,R))$ and the Khinchin family extends to include continuously a random variable $Y_R$ with mean $M_\psi<1$ and variance $\lim_{t\uparrow R} \sigma_{\psi}^2(t)=\sigma_{\psi}^2(R)<\infty$. Assume for simplicity that $Q_\psi=1$

In this case, we do not obtain a proper  asymptotic formula, just only that
\begin{equation}\label{eq:little oh con Mpsi menor que 1}\lim_{n \to \infty} A_n \frac{R^{n-1}}{\psi^n(R)} n^{3/2}=0\,.\end{equation}

For $A_n$ we have that
$$
A_n=\frac{1}{n} \textsc{coeff}_{[n-1]}(\psi^n)=
\frac{1}{2\pi} \frac{\psi^n(R)}{R^{n-1}}\frac{1}{n^{3/2}} \frac{1}{\sigma_{\psi}(R)}\, I_n
$$
where
$$I_n\triangleq \int\limits_{|\theta|\le \pi \sigma_\psi(R) \sqrt{n}} \E\big(e^{\imath \theta \breve{Y}_R /\sqrt{n}}\big)^n e^{\imath \theta \alpha(n)} d\theta\,, \quad \mbox{ for $n \ge 1$}\,,$$
and
$$\alpha(n)=\frac{n M_\psi -(n-1)}{\sigma_\psi(R)\sqrt{n}}\,, \quad \mbox{ for $n \ge 1$}\,.$$

We shall now verify that $I_n$ tends towards 0 as $n\to \infty$. We split $I_n$ as
$$
\begin{aligned}
	I_n&=\int\limits_{|\theta|\le \pi \sigma_\psi(R) \sqrt{n}} \Big(\E\big(e^{\imath \theta \breve{Y}_R /\sqrt{n}}\big)^n -e^{-\theta^2/2}\Big) e^{\imath \theta \alpha(n)} d\theta\\
	&+\int_{\R} e^{-\theta^2/2}e^{\imath \theta \alpha(n)} d\theta\\
	&-\int\limits_{|\theta|> \pi \sigma_\psi(R) \sqrt{n}} e^{-\theta^2/2}e^{\imath \theta \alpha(n)} d\theta\,.
\end{aligned}
$$
The third summand obviously tends to 0, while the first summand tends to 0 on account of Theorem \ref{teor:integral de diferencia caracteristicas}, the integral form of the Local Central Limit Theorem. The second summand may be written as
$$
\int_{\R} e^{-\theta^2/2}e^{\imath \theta \alpha(n)} d\theta=e^{-\alpha(n)^2/2}$$
to observe that since $M_\psi<1$, we have that $\lim_{n\to \infty} \alpha(n)=-\infty$, and thus that this second summand too converges to 0, as $n \to \infty$.

Minami in \cite[Section 2.3.2]{Minami} obtains a faster  rate of convergence (higher power of $n$) in \eqref{eq:little oh con Mpsi menor que 1} under the stronger assumption that $\lim_{t\uparrow R} \psi^{(k)}(t) <\infty$ for some $k \ge 3$.

\subsection{Coefficients of powers of solutions of Lagrange's equation}

We assume $Q_\psi=1$. We are interested now, see Meir-Moon \cite{MeirMoon},  in asymptotic formulas for the coefficients of  powers of the solution of Lagrange's equation.

For $q \ge 1$ and $n \ge 1$, the $n$-th coefficient of $g(w)^q$, which we will denote by $B_{n,q}$, is given by the exact formula
$$B_{n,q}=\frac{q}{n} \textsc{coeff}_{[n-q]}( \psi(z)^n)\, .$$

\smallskip

\noindent $\bullet$ \quad For $q\ge 1$ fixed, Theorem \ref{teor:large powers k/n to L} gives, under the assumption $M_\psi>1$ and with $\tau$ given by $m_\psi(\tau)=1$, that
$$
B_{n,q}\sim  \frac{q}{\sqrt{2\pi}} \frac{\tau^q}{\sigma_\psi(\tau)} \frac{1}{n^{3/2}} \Big(\frac{\psi(\tau)}{\tau}\Big)^n\, , \quad \mbox{as $n\to \infty$ }.
$$

\noindent $\bullet$ \quad If $q$ and $n$ are related by $q=\alpha n+\beta \sqrt{n} +o(\sqrt{n})$ as $n \to \infty$, where $\alpha \in[0,1)$ and $\beta \in \R$, then for $\tau_\alpha$ given by $m_\psi(\tau_\alpha)=1-\alpha$, we obtain, analogously to \eqref{eq:formula_limit},  that
$$B_{n,q}\sim e^{-\beta^2/(2\sigma_\psi^2(\tau_\alpha))}\frac{q}{\sqrt{2\pi}} \frac{\tau_\alpha^q}{\sigma_\psi(\tau_\alpha)} \frac{1}{n^{3/2}} \Big(\frac{\psi(\tau_\alpha)}{\tau_\alpha}\Big)^n\,,\quad \mbox{as $n \to \infty$}\,.$$
Observe that
$$\sigma_\psi(\tau_\alpha)=\alpha(1-\alpha)+\tau^2_\alpha\frac{ \psi^{\prime\prime}(\tau_\alpha)}{\psi(\tau_\alpha)}\,.$$

%

\smallskip

\subsubsection{Coefficients of functions of solutions of Lagrange's equation} In a more general setting, let $H$ be a power series with nonnegative coefficients.

Assume further that the radius of convergence of $H$ is at least the radius of convergence $R$ of $\psi$ and also that $M_\psi>1$. Let $\tau$ be such that $m_\psi(\tau)=1$.

It follows from formulas \eqref{eq: Lagrange_general} and \eqref{eq:formula_limit} that the coefficients of $H(g(z))$, where $g(z)$ is the solution of Lagrange's equation with data $\psi$, satisfy:
$$
\textsc{coeff}_{[n]} (H(g(z))) \sim \frac{1}{\sqrt{2\pi}}\frac{H^\prime(\tau) \tau}{\sigma_\psi(\tau)}\frac{1}{n^{3/2}} \Big(\frac{\psi(\tau)}{\tau}\Big)^n\, , \quad \mbox{as $n \to \infty$}\,.
$$

\section{Probability generating functions}\label{section:prob gen functions}

Let $\psi$ be the probability generating function of a random variable $X$.

We assume  for convenience that $\psi(0)\neq 0$ and that $\psi^\prime(0)\neq 0$,  and further that $Q_\psi=1$,  that $M_\psi=+\infty$ and  that the radius of convergence of $\psi$ is $R>1$. This last assumption implies, in particular, that the variable $X$ has finite moments of all orders.

\smallskip

Let $F$ be the fulcrum of $\psi$, see Section \ref{section:auxiliary F}, given by $F(z)=\ln \psi(e^z)$, in a region containing $[0,R_\psi)$. Recall  that in  terms of the fulcrum $F$ we have that $m_\psi(t)=F^\prime(s)$ and $\sigma^2_\psi(t)=F^{\prime\prime}(s)$, where $s$ and $t$ are related by $e^s=t$.

The function $e^{F(z)}=\psi(e^z)=\sum_{n=0}^\infty b_n e^{n z}$,  which is holomorphic in a neighborhood of $z=0$,  is the \textit{moment generating function} of $X$.

\smallskip

If $X_1, X_2, \ldots$ are independent copies of $X$, then
$$\textsc{coeff}_{[k]}(\psi(z)^n)=\P\left(\frac{X_1+\dots+X_n}{n}=\frac{k}{n}\right)\,.$$

Let $m_\psi(\tau_n)=k/n$ and let $s_n$ be given by $e^{s_n}=\tau_n$. Observe that $\psi(\tau_n)=e^{F(s_n)}$ and that $\sigma^2_\psi(\tau_n)=F^{\prime\prime}(s_n)$.

\

In terms of the fulcrum $F$ (or the $\ln$ of the moment generating function),  we have that if $k\asymp n$, see Theorem \ref{teor:asymptotic powers with k comp n} or $k/n\to 0$, see Theorem \ref{teor:asymptotic powers with k=o(n)}, that
$$\begin{aligned}
	\P\left(\frac{X_1+\dots+X_n}{n}=\frac{k}{n}\right) &\sim\frac{1}{\sqrt{2\pi n} \, \sigma_\psi(\tau_n)}\frac{\psi(\tau_n)^n}{\tau_n^k} \\&=\frac{1}{\sqrt{2\pi n\, F^{\prime\prime}(s_n)}}e^{n [F(s_n)-s_n (k/n)]}\,, \quad \mbox{as $n \to \infty$}\,.\end{aligned}$$

If $\psi$ is uniformly Gaussian, the asymptotic formula above holds also if $k/n\to \infty$.

\subsection{Lagrangian distributions (and Galton-Watson processes)}\label{section:lagrangian distributions}

We now turn our attention to Lagrangian probability distributions. We refer to \cite{Sibuya} for a neat presentation of the basic theory of these  probability distributions. See also \cite{Consul}, for a comprehensive treatment, and \cite{Pakes}.

We start with a probability generating function $\phi$ of a variable $Y$ taking values in $\{0,1, \ldots\}$.  We assume that $\phi$ is in $\K$.

This variable  $Y$ generates a cascade or Galton-Watson random process starting (initial stage) with a single  individual. The variable $Y$ gives the random number of immediate descendants, the offsprings of each individual in every generation.

The random number of individuals in generation $n$ is denoted by $G_n$; the generation $0$ is the initial stage. We have $G_0\equiv 1$ and
$$(\star) \quad G_{n+1}=\sum_{j=1}^{G_n} Y_{n,j}\,, \quad \mbox{for $n \ge 0$}\,,$$
where the $Y_{n,j}$ are independent copies of $Y$.

Let us denote $\phi^\prime(1)\triangleq \lim_{t\uparrow 1} \phi^\prime(t)=\E(Y)$. Observe that $\phi^\prime(1)=m_\phi(1)$.

The total progeny $Z$ of the single individual  of the generation 0 (including itself) is the random variable $Z=\sum_{n=0}^\infty G_n$. This variable $Z$ could take the value $\infty$, but it is a proper random variable (i.e., $\P(Z<\infty)=1$) if and only if $\psi^\prime(1)\le 1$.

Assume thus that  $\psi^\prime(1)\le 1$. The probability generating function $g$ of the total progeny $Z$ is actually the solution of Lagrange's equation with data $\phi$.

\smallskip

Furthermore,  let $f$ be a power series with nonnegative coefficients  which is the probability generating function of a random variable $X$ taking values in $\{0,1, \ldots\}$.

We  enhance the process by allowing the size of the   initial stage (generation $0$) to be randomly chosen  following the distribution of $X$. Thus $G_0=X$ and the evolution is determined by the recurrence $(\star)$. The total progeny $Z$ including the individuals of the initial stage has probability generating function $f\circ g$.

This size of progeny $Z$ is said to \textit{follow a Lagrangian distribution} $\mathcal{L}(\phi, f)$; $\phi, f$ are called the generators of $\mathcal{L}(\phi, f)$.

If $f(z)\equiv z$, then $X \equiv 1$, and there is (deterministically) a single individual in the initial generation.

We have, see \eqref{eq: Lagrange_general}, that
$$\P(Z=n)=\textsc{coeff}_{[n]} (f(g(z)))=\frac{1}{n} \textsc{coeff}_{[n{-}1]}(f^\prime(z) \phi^n(z))\,, \quad \mbox{for $n \ge 1$}\,,$$
and $\P(Z=0)=f(0)$.

\medskip

\subsection{Lagrangian distributions and Khinchin families}

We will have two ingredients: a Khinchin family of offspring probability distributions and a Khinchin  family
of   probability distribution for the initial distribution.

\medskip

\noindent (A) \quad Let $\psi(z)=\sum_{n=0}^\infty b_n z^n$ be a power series in $\K$ with radius of convergence $R$. We assume from  the outset that $Q_\psi=1$. We let $(Y_t)_{t \in [0,R)}$

Assume that either $M_\psi>1$ and then we let $\tau \in (0,R)$ be such  that $m_\psi(\tau)=1$ or $M_\psi=1$ and $R<\infty$ and $\lim_{t\uparrow R} \sigma_\psi(t)=\sigma_\psi(R)<\infty$ and then we let $\tau=R$.

For each $t \in (0,\tau]$, we let $\psi_t$ denote the power series $\psi_t(z)=\psi(tz)/\psi(t)$. This $\psi_t$ is the  probability generating function of $Y_t$,   and besides, since $t\le \tau$, we have that
$$
(\dag) \quad \psi_t^\prime(1)=m_\psi(t)\le 1\,.$$

For each $t \in (0,\tau]$, we let $g_t$ be the solution of  Lagrange's equation with data $\psi_t$:
\begin{align}\label{eq:lagrange_g_t}
	g_t(z)=z \psi_t(g_t(z))\,.
\end{align}
Because of $(\dag)$  and the discussion above in Section \ref{section:lagrangian distributions}, $g_t$ is the probability generating function of the random distribution of the total progeny  of a single individual with  offspring distribution $Y_t$.

We let $g$ denote the solution of Lagrange's equation with data $\psi$, then we may write each $g_t$ in terms of $g$ as follows

\begin{lem}\label{lema:lagrange_g_t_probgen}
	With the notations above
	$$g_t(z)=\frac{1}{t} g\big((t/\psi(t))z\big)\,, \quad \mbox{for $t \le \tau$}\,.$$
\end{lem}
\begin{proof}
	This is a consequence of the uniqueness of solution of Lagrange's equation. Let $\widetilde{g}_t(z)= \frac{1}{t} g\big((t/\psi(t))z\big)$. Now,
	$$\begin{aligned}z\psi_t(\widetilde{g}_t(z))&=\frac{z}{\psi(t)}\psi(t \, \widetilde{g}_t(z))=
		\frac{z}{\psi(t)} \psi(g\big((t/\psi(t))z\big)\\&=\frac{1}{t}\frac{zt}{\psi(t)} \psi(g\big((t/\psi(t))z\big)=
		\frac{1}{t}g\big((t/\psi(t))z\big)\\&=\widetilde{g}_t(z)\,.\end{aligned}$$
	Uniqueness gives that $\widetilde{g}_t(z)=g_t(z)$, as claimed.\end{proof}

\begin{remark}\label{remark: rescaled_Khinchin_Lagrange}The $g_t$ are \textit{not} the probability generating functions of a Khinchin family, but if we change parameters and substitute $t$ by $g(u)$ and let
	$$\widetilde{g}_u(z)=g_{g(u)}(z)$$
	we have that $\widetilde{g}_u(z)=g(uz)/g(u)$. The power series $g$ is not in $\K$ but $g(z)/z$ is in $\K$, since $g^\prime(0)=\psi(0)>0$. If $(W_u)$ is the Khinchin family of $g(z)/z$ then $\widetilde{g}_u$ is the probability generating functions of $W_u+1$.\end{remark}

\medskip

\noindent (B) \quad Next, we let $f(z)=\sum_{n=0}^\infty a_n z^n$ be a nonconstant power series  with nonnegative coefficients and radius of convergence $S>0$. We do not require $f$ to be in  $\K$; in fact $f(z)=z^m$, for integer $m \ge 1$, is particularly relevant.

For $s \in (0,S)$ we denote by $f_s$ the power series
$$
f_s(z)=\frac{f(sz)}{f(s)}
$$ which is the probability generating function of a random variable $X_s$, say.

\begin{remark}\label{remark: prob_gen_Khinchin}	If the power series $f$ were not in $\K$, then since $f$  is not constant, there should exist at least one index $m\ge 1$ so that $a_m\neq 0$. Thus only two cases could occur.
	In the first case, $f(z) = a_mz^m$, for some integer  $m \geq 1$. And in the second; there is  $\phi \in \K$ and an integer $l \geq 0$ so that $f(z) = z^l \phi(z)$.
	
	In the first case $(X_s)$ is not a Khinchin family, but $X_s \equiv m$, for all $s \in (0,S)$, and therefore there are exactly $m \geq 1$ nodes in the first generation; this is the deterministic case $f_s(z) = z^m$.
	
	In the second case $(X_s)_{s \in [0,S)}$ is a shifted Khinchin family. If $f(0) >0$, then  $(X_s)_{s \in [0,S)}$  is a proper Khinchin family.
\end{remark}

Now, for $s\in (0,S)$ and $t \in (0,\tau]$, the composition $f_s(g_t(z))$ is the probability generating function of the total progeny $Z_{s,t}$ of a Galton-Watson process with initial distribution $f_s$ and offspring distribution $g_t$. The progeny $Z_{s,t}$  has Lagrange distribution $\mathcal{L}( \psi_t, f_s)$.


For radius $u>0$ such that $us<S$ and $ut\le \tau$ and appealing to \eqref{eq: Lagrange_general} we may write
$$
\begin{aligned}
	\P(Z_{s,t}=n)&=\textsc{coeff}_{[n]}\big(f_s(g_t(z))\big)=\frac{1}{n}\textsc{coeff}_{[n{-}1]}\big(f_s^\prime(z) \psi_t^n(z)\big)\\
	&=\frac{s}{f(s)} \frac{\psi(tu)^n}{\psi(t)^n}\frac{1}{u^{n{-}1}}\frac{1}{n}\, \frac{1}{2\pi}\int\limits_{|\theta|\le \pi}
	f^\prime(su e^{\imath \theta}) \frac{\psi(tu e^{\imath \theta})^n}{\psi(tu )^n}e^{-\imath (n{-}1) \theta} d\theta\,.
\end{aligned}
$$

If the parameter $s$ is further restricted to $s\tau<t S$, we may take $u=\tau/t$ in the expression above and write
$$
\begin{aligned}
	\P(Z_{s,t}=n)&=\frac{1}{2\pi}\frac{s}{f(s)} \frac{\psi(\tau)^n}{\psi(t)^n} \left(\frac{t}{\tau}\right)^{n{-}1}\frac{1}{n}\int\limits_{|\theta|\le \pi}
	f^\prime\left(\frac{s\tau}{t}e^{\imath \theta}\right) \E(e^{\imath Y_\tau \theta})^n e^{-\imath (n{-}1) \theta}d\theta
	\\
	&=\frac{1}{2\pi}\frac{s}{f(s)} \frac{\psi(\tau)^n}{\psi(t)^n}\Big(\frac{t}{\tau}\Big)^{n{-}1}\frac{1}{n^{3/2}}
	\frac{1}{\sigma_\psi(\tau)}\\&
	\int\limits_{|\theta|\le \pi \sigma_\psi(\tau) \sqrt{n}}
	f^\prime\left(\frac{s\tau}{t}e^{\imath \theta/(\sigma_\psi(\tau)\sqrt{n})}\right)
	\E(e^{\imath \breve{Y}_\tau\theta/\sqrt{n}})^n e^{\imath \theta/(\sigma_\psi(\tau)\sqrt{n})} d\theta
\end{aligned}
$$

For $\theta$ fixed, we have that
$$\lim_{n \to \infty} f^\prime\left(\frac{s\tau}{t}e^{\imath \theta/(\sigma_\psi(\tau)\sqrt{n})}\right) e^{\imath \theta/(\sigma_\psi(\tau)\sqrt{n})} =f^\prime\left(\frac{s\tau}{t}\right)
$$ and
$$\Big| f^\prime\left(\frac{s\tau}{t}e^{\imath \theta/(\sigma_\psi(\tau)\sqrt{n})}\right) e^{\imath \theta/(\sigma_\psi(\tau)\sqrt{n})} \Big| \le f^\prime\left(\frac{s\tau}{t}\right)\,.
$$
And thus taking into account the integral form of the Local Central Limit Theorem, Theorem \ref{teor:integral de diferencia caracteristicas},   we deduce that
$$
\lim_{n \to \infty}\int\limits_{|\theta|\le \pi \sigma_\psi(\tau) \sqrt{n}}
f^\prime\left(\frac{s\tau}{t}e^{\imath \theta/(\sigma_\psi(\tau)\sqrt{n})}\right)
\E(e^{\imath \breve{Y}_\tau\theta/\sqrt{n}})^n e^{\imath \theta/(\sigma_\psi(\tau)\sqrt{n})} d\theta=\sqrt{2\pi} f^\prime\left(\frac{s\tau}{t}\right)$$
and, consequently, that
\begin{equation}\label{eq:asymptotics Lagrangian distributions}
	\P(Z_{s,t}=n)\sim \frac{1}{\sqrt{2\pi}}\frac{s}{f(s)} \frac{\psi(\tau)^n}{\psi(t)^n}\Big(\frac{t}{\tau}\Big)^{n{-}1}\frac{1}{n^{3/2}}
	\frac{1}{\sigma_\psi(\tau)} f^\prime\left(\frac{s\tau}{t}\right)\, , \quad \mbox{as $n \to \infty$}\,,\end{equation}
as long as $s\tau<t S$, which amounts to no restriction if $S=+\infty$.
\smallskip

\noindent (C) \quad As an illustration,  consider the case where $\psi(z)=e^z$ and $f(z)=z^j$, for some integer $j \ge 1$.

In this case, $R=S=\infty$, $m_\psi(t)=t$ and $\sigma_\psi^2(t)=t$. Also $M_\psi=\infty$ and $\tau=1$.

For $t\le 1=\tau$, we have that $\psi_t(z)=e^{t(z-1)}$ and for $s<\infty $, we have that $f_s(z)=z^j$. Observe that for any $s$, $f_s$ is the probability generating function of the constant $j$.

For $0<t \le 1$ and $0< s<\infty$, the variable $Z_{s,t}$ is the total progeny of a Galton-Watson process, where the initial generation consists of exactly $j$ individuals and the  offspring of each individual is given by a Poisson variable of parameter $t$. This distribution, $\mathcal{L}(e^{t(z-1)}, z^j)$, is the Borel-Tanner distribution with parameters $t$ and $j$. The case $j=1$ is the Borel distribution. See \cite{Consul}, \cite{Pitman} and \cite{Sibuya}, and also the original sources \cite{Borel}, \cite{HaightBreuer} and \cite{Tanner}.

Using \eqref{eq:asymptotics Lagrangian distributions}, we deduce that
\begin{equation}\label{eq:formula uno Zst}
\P(Z_{s,t}=n)\sim \frac{j}{\sqrt{2\pi}} \frac{1}{n^{3/2}} {t^{n-j}} e^{n(1-t)}\,,
\quad \mbox{as $n \to \infty$}\,.
\end{equation}
In fact, for the Borel-Tanner distribution with parameters $t$ and $j$ we have the exact formula
\begin{equation}\label{eq:borel-tanner}\P(Z_{s,t}=n)=\frac{j}{n} \frac{e^{-tn} (tn)^{n{-}j}}{(n-j)!}\,, \quad \mbox{for $n \ge j$}\,.\end{equation}
The asymptotic formula \eqref{eq:formula uno Zst} follows then from Stirling's formula.

\medskip

\noindent (D) \quad Consider now the case $\psi(z)=e^z$ and $f(z)=e^z$, so that  $R=S=\infty$, $m_\psi(t)=t$ and $\sigma_\psi^2(t)=t$ and, also, $M_\psi=\infty$ and $\tau=1$.

For $t\le 1=\tau$, we have that $\psi_t(z)=e^{t(z-1)}$ and for $s<\infty $, we have that $f_s(z)=e^{s(z-1)}$.

For $0<t \le 1$ and $0< s<\infty$, the variable $Z_{s,t}$ is the total progeny of a Galton-Watson process, where the size of the  initial generation is drawn from a Poisson distribution of parameter $s$  individuals and the  offspring of each individual is given by a Poisson variable of parameter $t$. This distribution, in Lagrangian distribution notation, is  $\mathcal{L}(e^{t(z-1)}, e^{s(z-1)})$.

From \eqref{eq:asymptotics Lagrangian distributions}, we have  that
\begin{equation}\label{eq:formula dos Zst}\P(Z_{s,t}=n)\sim \frac{1}{\sqrt{2\pi}} e^{s/t-s} s t^{n{-}1} e^{n(1-t)}\frac{1}{n^{3/2}}\,, \quad \mbox{as $n \to \infty$}\,.\end{equation}

Conditioning on the size of the initial generation and using \eqref{eq:borel-tanner}, we deduce for the distribution $\mathcal{L}(e^{t(z-1)}, e^{s(z-1)})$ that
$$\P(Z_{s,t}=n)=\frac{1}{n!} e^{-tn-s} (tn+s)^{n{-}1} s\,, \quad \mbox{para $n \ge 1$}\,.$$
See \cite{Sibuya}.
The asymptotic formula \eqref{eq:formula dos Zst} follows then from Stirling's formula.

\medskip

\noindent (E) \quad \textit{Back to Galton-Watson.}
Assume that  $\psi$ is a probability generating function and that $f(z)=z$.
We assume that $\psi^{\prime}(1)\le 1$, so that the total progeny $Z$ of the Galton-Watson process starting with a single individual and with offspring distribution $\psi$ is a (proper) random variable.

Assume that $M_\psi>1$ or $M_\psi=1$ with $R<\infty$ and $\sigma_\psi(R)<\infty$. In the first case we take $\tau\in (0,R)$ such that $m_\psi(\tau)=1$. Since $\psi^{\prime}(1)\le 1$, we have that $m_\psi(1)=1$, and since $m_\psi$ is increasing we see that $1\le \tau <R$. In the second case we take $\tau=R$; observe that $R\ge 1$.

With $t=1$ (so that $\psi_1(z) \equiv \psi(z)$) and $s=1$ (an immaterial choice since $f(z)=z$), we have
$$\P(Z=n)\sim \frac{1}{\sqrt{2\pi}}\frac{\tau}{\sigma_{\psi}(\tau)} \Big(\frac{\psi(\tau)}{\tau}\Big)^n \frac{1}{n^{3/2}}\,, \quad \mbox{as $n \to \infty$}\,.$$ This is Theorem \ref{teor:Otter-Meir-Moon} applied to the solution of Lagrange's equation with data $\psi$. Recall that $\sigma_{\psi}^2(\tau) = {\tau^2 \psi^{\prime \prime}(\tau)}/{\psi(\tau)}$, see formula (\ref{eq: varianza_tau}).

\subsubsection{Limit cases} By limit cases we mean  $t=\tau$ and $s \to S$, with $S<\infty$. See \cite{Mutafchiev} for related results.

For $t=\tau$ and $s<S$ we have  the exact formula
$$(\flat)\qquad \begin{aligned}
	\P(Z_{s,\tau}=n)
	&=\frac{1}{2\pi}\frac{s}{f(s)} \frac{1}{n^{3/2}}
	\frac{1}{\sigma_\psi(\tau)}\\&
	\int\limits_{|\theta|\le \pi \sigma_\psi(\tau) \sqrt{n}}
	f^\prime\big(se^{\imath \theta/(\sigma_\psi(\tau)\sqrt{n})}\big)
	\E(e^{\imath \breve{Y}_\tau\theta/\sqrt{n}})^n e^{\imath \theta/(\sigma_\psi(\tau)\sqrt{n})} d\theta
\end{aligned}
$$
and the asymptotic formula.
$$
\P(Z_{s,\tau}=n)\sim \frac{1}{\sqrt{2\pi}} m_f(s) \frac{1}{\sigma_\psi(\tau)} \frac{1}{n^{3/2}}
\, , \quad \mbox{as $n \to \infty$}\,,$$
where with a slight abuse of notation we have written ${sf^\prime(s)}/{f(s)} =m_f(s)$, see  Remark \ref{remark: prob_gen_Khinchin}.

To consider  $s \to S$, we first rewrite $(\flat)$ as
$$(\flat\flat)\qquad \begin{aligned}
	\P(Z_{s,\tau}=n)&=\frac{1}{2\pi} \frac{1}{n^{3/2}}{m_f(s)}
	\frac{1}{\sigma_\psi(\tau)}\\&\int\limits_{|\theta|\le \pi \sigma_\psi(\tau) \sqrt{n}}
	\dfrac{f^\prime\big(se^{\imath \theta/(\sigma_\psi(\tau)\sqrt{n})}\big)}{f^\prime(s)}
	\E(e^{\imath \breve{Y}_\tau\theta/\sqrt{n}})^n e^{\imath \theta/(\sigma_\psi(\tau)\sqrt{n})} d\theta\end{aligned}$$

If $\sum_{n=0}^\infty n a_n S^n<\infty$, then $f^\prime$ extends continuously to $\cl(\D(0,S))$, and by appealing to Theorem \ref{teor:integral de diferencia caracteristicas}, we readily see that
$$
\lim_{\substack{s\uparrow S;\\n \to \infty}} \frac{n^{3/2}}{m_f(s)}\P(Z_{\tau,s}=n)=\frac{1}{\sigma_\psi(\tau)\sqrt{2\pi} }\,.$$

More generally, if $s_n<S$ and $n$ are such that
$$
(\sharp) \quad\lim_{n \to \infty}\frac{f^\prime(s_n e^{\imath \phi/\sqrt{n}})}{f^\prime(s_n)}=1\,, \quad \mbox{for each $\phi \in \R$}\,,$$
then
$$
\lim_{n \to \infty} \frac{n^{3/2}}{m_f(s_n)}\P(Z_{\tau,s_n}=n)=\frac{1}{\sigma_\psi(\tau)\sqrt{2\pi} }\,.$$

For instance for $f(z)=1/(1-z)$, condition $(\sharp)$ is satisfied if $s_n$ and $n$ are related so that $\lim_{n \to \infty}(1-s_n)\sqrt{n}=\infty$.

\section{
	Appendix. Uniformly Gaussian  Khinchin families}\label{section:uniformly Gaussian}

Let $f(z)=\sum_{n=0}^\infty b_n z^n$ be a power series in $\K$ with radius of convergence $R>0$ and let $(X_t)_{t \in [0,R)}$ be its Khinchin family.

\medskip

We have encountered two integral convergence results: the integral form of the Local Central Limit Theorem \ref{teor:integral de diferencia caracteristicas} and the notion of strongly gaussian power series, definition \ref{defin:strongly gaussian}.

\medskip

\noindent $\bullet$ Assume that $Q_f=\gcd\{n \ge 1: b_n >0\}=1$. \textit{For each fixed $t \in (0,R)$}, the normalized variable $\breve{X}_t$ is a lattice random variable with gauge function $1/\sigma_f(t)$, since $Q_f=1$. Because of Theorem \ref{teor:integral de diferencia caracteristicas} we then have that
$$(\dag)\qquad \lim_{n \to \infty} \int\limits_{|\theta|\le \pi \sigma_f(t) \sqrt{n}}
\Big|\E\Big(e^{\imath \theta\breve{X}_t/\sqrt{n}}\Big)^n -e^{-\theta^2/2}\Big|\, d\theta=0\, .$$

\

\noindent $\bullet$  If the Khinchin family $(X_t)_{t \in [0,R)}$ is strongly Gaussian, then we have \textit{for each fixed $n \ge 1$}, that
$$(\ddag) \qquad \lim_{t \uparrow R} \int\limits_{|\theta|\le \pi \sigma_f(t) \sqrt{n}}
\Big|\E\Big(e^{\imath \theta\breve{X}_t/\sqrt{n}}\Big)^n -e^{-\theta^2/2}\Big|\, d\theta=0\, .$$

This fact follows from the bound
$$\begin{aligned}
	\int\limits_{|\theta|\le \pi \sigma_f(t) \sqrt{n}}
	\Big|\E\Big(e^{\imath\theta \breve{X}_t/\sqrt{n}}\Big)^n -e^{-\theta^2/2}\Big|\, d\theta&=\sqrt{n}\, \int\limits_{|\varphi|\le \pi \sigma_f(t) }
	\Big|\E\Big(e^{\imath \varphi\breve{X}_t}\Big)^n -e^{-\varphi^2 n/2}\Big|\, d\varphi
	\\&\le n^{3/2} \int\limits_{|\varphi|\le \pi \sigma_f(t) }
	\Big|\E\Big(e^{\imath \varphi\breve{X}_t}\Big) -e^{-\varphi^2/2}\Big|\, d\varphi\, ,\end{aligned}$$
where, after the change of variables $\theta=\varphi \sqrt{n}$,  we have used that for complex numbers $z,w$ such that $|z|,|w|\le 1$ we have that $|z^n-w^n|\le n |z-w|$.

\

Thus, for a Gaussian power series the integral
$$I_n(t)=\int\limits_{|\theta|\le \pi \sigma_f(t) \sqrt{n}}
\Big|\E\Big(e^{\imath\theta \breve{X}_t/\sqrt{n}}\Big)^n -e^{-\theta^2/2}\Big|\, d\theta
$$
converges to 0 as $n \to \infty$ with $t$ fixed and as $t \uparrow R$ with $n$ fixed.

\

Power series $f$ in $\K$ with Khinchin family $(X_t)_{t \in [0,R)}$ are power series for which $(\dag)$ and $(\ddag)$ hold \textit{simultaneously} in the sense that the involved integrals converges to 0, as $n \to \infty$ or $t\uparrow R$.

\begin{defin}\label{defin:uniformly Gaussian} A power series $f$ and its Khinchin  family $(X_t)_{t \in [0,R)}$ are called \textit{uniformly Gaussian} if the following two conditions are satisfied:
	$$a)\quad \lim_{t \uparrow R} \sigma_f(t)=\infty\quad \mbox{and} \quad b) \quad \lim_{[n \to \infty \, \vee \,  t \uparrow R]} \int\limits_{|\theta|\le \pi \sigma_f(t) \sqrt{n}}
	\Big|\E\Big(e^{\imath \theta\breve{X}_t/\sqrt{n}}\Big)^n -e^{-\theta^2/2}\Big|\, d\theta=0\, .$$
\end{defin}

By $[n{\to} \infty \, \vee \,  t{\uparrow} R]$ we mean that $1 \le n \to \infty$ or (inclusive) $0\le t_0\le t \uparrow R$.
The restriction $t>t_0$ is there to exclude  the possibility of $t \to 0$ and $n \to \infty$, simultaneously.

\

By fixing $n=1$ and letting $t \uparrow R$, we observe that uniformly Gaussian power series are  strongly Gaussian. In particular, if $f$ is a  uniformly Gaussian power series then
$$M_f=\lim_{t \uparrow R} m_f(t)=\infty\,.$$
Moreover, the coefficients $b_n$ of $f$ satisfy Hayman's asymptotic formula and in particular $Q_f=1$.
\

Let us verify that the exponential $f(z)=e^z$ is uniformly Gaussian. We have $\sigma_f(t)=\sqrt{t}$, for $t >0$ and so $\sigma_f(t)\sqrt{n}=\sigma_f(nt)$, for $t >0$ and $n \ge 1$.
For the characteristic function of the normalized variable $\breve{X}_t$ we have that
$$
\E\big(e^{\imath \theta \breve{X}_t}\big)=\exp\big(t\, \big(e^{\imath \theta/\sqrt{t}}-1-\imath \theta/\sqrt{t}\big)\Big)\, ,$$
and thus that
$$\E\big(e^{\imath \theta \breve{X}_t/\sqrt{n}}\big)^n=\E\big(e^{\imath \theta \breve{X}_{nt}}\big) \,.$$
Therefore,
$$
\int\limits_{|\theta|\le \pi \sigma_f(t) \sqrt{n}}
\Big|\E\Big(e^{\imath \theta\breve{X}_t/\sqrt{n}}\Big)^n -e^{-\theta^2/2}\Big|\, d\theta=
\int\limits_{|\theta|\le \pi \sigma_f(n t)}
\Big|\E\Big(e^{\imath \theta\breve{X}_{n t}}\Big) -e^{-\theta^2/2}\Big|\, d\theta\,.
$$
Since $e^z$ is strongly Gaussian, this integral tends to 0 as $(nt)\to \infty$, and thus as
$[n{\to} \infty \, \vee \,  t{\uparrow} +\infty]$. Observe that in this case $R=+\infty$.

\

The notion of uniformly Hayman power series  which we are about to introduce generalizes that of Hayman power series, see Section \ref{section:Hayman class}. We will verify shortly that uniformly Hayman power series are uniformly Gaussian, much like power series in the Hayman class are strongly gaussian. Theorem \ref{teor:exponentials uniformly Hayman} provides us with an ample class of power series which are uniformly Hayman.

\begin{defin}\label{defin:uniformly Hayman}
	Let $f(z)$ be a power series in $\K$ with radius of convergence $R>0$ and let $(X_t)_{t \in [0,R)}$ be its Khinchin family.

	We say that $f(z)$ and  $(X_t)_{t \in [0,R)}$ are \textit{uniformly Hayman} if for each $n \ge 1$ and $t \in (0,r)$ there exists   $h(n,t) \in (0, \pi)$ (called \textit{cuts}) such that the following requirements  are satisfied:
	\begin{align}
		\label{eq:uniformly Hayman major}\qquad  &  \sup_{|\theta|\le h(n,t) \sigma_f(t) \sqrt{n}} \Big|\E(e^{\imath \theta \breve{X}_t /\sqrt{n}})^n e^{\theta^2/2}-1\Big| \rightarrow 0, \quad \mbox{as $[n{\to}\infty \, \vee \,  t{\uparrow}R]$}\,,
		\\[9pt]
		\label{eq:uniformly Hayman minor}\qquad & \sqrt{n} \sigma_f(t) \, \sup_{ h(n,t) \sigma_f(t) \sqrt{n}\le |\theta| \le \pi \sigma_f(t) \sqrt{n}} \big|\E(e^{\imath \theta \breve{X}_t/\sqrt{n}})^n\big| \rightarrow 0, \quad \mbox{as $[n {\to} \infty \, \vee \,  t {\uparrow} R]$}\,,
		\\[9pt]
		\label{eq:uniformly Hayman variance}\qquad &\lim_{t \to R} \sigma_f(t)=\infty\, .
	\end{align}
\end{defin}

\

Condition \eqref{eq:uniformly Hayman minor} may be written equivalently as
\begin{equation}
	\label{eq:uniformly Hayman minor bis} \qquad  \sqrt{n} \sigma_f(t) \, \sup_{ h(n,t) \le |\theta| \le \pi}\big|\E(e^{\imath \theta X_t})^n\big| \rightarrow 0, \quad \mbox{as $[n {\to} \infty \, \vee \,  t {\uparrow} R]$}\,.\end{equation}

\

As announced,
\begin{theor}
	Uniformly Hayman power series are uniformly Gaussian.
\end{theor}

The proof below is analogous to the proof of Theorem \ref{teor:hayman implica baez} of Section \ref{section:Hayman class} which claims that Hayman power series are strongly Gaussian.

\begin{proof} Denote $\theta(n,t)=h(n,t) \sigma_f(t) \sqrt{n}$. First we show,
	that
	\begin{equation}\label{eq:theta(n,t) to infty}
		\theta(n,t) \to \infty\, , \quad \mbox{as $[n {\to} \infty \, \vee \,  t {\uparrow} R]$}\,.
	\end{equation}

	Abbreviate $\widehat{\theta}=\theta(n,t)$.
	By \eqref{eq:uniformly Hayman major}, we have that
	$$\E(e^{\imath \widehat{\theta} \breve{X}_t/\sqrt{n}}) e^{\widehat{\theta}^2/2} \to 1\, , \quad
	\mbox{as $[n {\to} \infty \, \vee \,  t {\uparrow} R]$}\,,$$
	while, from  \eqref{eq:uniformly Hayman minor} we obtain that
	$$\sqrt{n} \sigma_f(t) \E\big(e^{\imath \widehat{\theta} \breve{X}_t/\sqrt{n}}\big)^n \to 0\, , \quad \mbox{as $[n {\to} \infty \, \vee \,  t {\uparrow} R]$}\,.$$
	From these two limits we deduce that
	$$\sqrt{n}\sigma_f(t) e^{-\widehat{\theta}^2/2} \to 0\, , \quad \mbox{as $[n {\to} \infty \, \vee \,  t {\uparrow} R]$}\,,$$
	and thus, since $\sqrt{n}\sigma_f(t) \to \infty$ as   $[n {\to} \infty \, \vee \,  t {\uparrow} R]$ we deduce that
	$$\widehat{\theta} \to \infty \, , \quad \mbox{as $[n {\to} \infty \, \vee \,  t {\uparrow} R]$}\,.$$
	
	\
	
	Denote by $A(n,t), B(n,t)$, respectively,  the supremum in \eqref{eq:uniformly Hayman major} and \eqref{eq:uniformly Hayman minor}.
	
	\
	
	We bound
	$$\begin{aligned} \int\limits_{|\theta|\le h(n,t) \sigma_f(t) \sqrt{n}} &\Big|\E(e^{\imath \theta \breve{X}_t /\sqrt{n}})^n- e^{-\theta^2/2}\Big|\, d\theta\\=\int\limits_{|\theta|\le h(n,t) \sigma_f(t) \sqrt{n}} &\Big|\E(e^{\imath \theta \breve{X}_t /\sqrt{n}})^n e^{\theta^2/2}-1\Big|  e^{-\theta^2/2}\, d \theta \\&\le  A(n,t)\sqrt{2\pi}\, ,\end{aligned}$$
	and
	$$\begin{aligned}\int\limits_{h(n,t) \sigma_f(t) \sqrt{n} \le |\theta|\le \pi \sigma_f(t)\sqrt{n}}&\Big|\E(e^{\imath \theta \breve{X}_t /\sqrt{n}})- e^{-\theta^2/2}\Big|\, d\theta \\ \\\le &2\pi \sigma_f(t) \sqrt{n}B(n,t)+ \int_{|\theta|\ge h(n,t) \sigma_f(t) \sqrt{n}} e^{-\theta^2/2} \, d \theta\, .\end{aligned}$$
	
	These two bounds and conditions \eqref{eq:uniformly Hayman major} and \eqref{eq:uniformly Hayman minor} combined with \eqref{eq:theta(n,t) to infty} give the result.
	
\end{proof}

\subsection{Uniformly Hayman exponentials}\label{section:exponentials uniformly Hayman}

Let   $g(z)=\sum_{n=0}^\infty b_n z^n$, with  $b_n \ge 0$, for $n \ge 0$, and radius of convergence $R>0$. Let $f\in \K $ be given by $f=e^g$.

Exponentials of power series with positive coefficients are very relevant, in particular, in Combinatorics since they codify (most) generating functions of the set construction, which includes among them generating functions of partitions of many sorts.

One of the main results of \cite{K_dos}, see \cite[Theorem 4.1]{K_uno} and \cite[Theorem F]{K_dos} gives conditions on the powers series $g$ that guarantees that $f=e^g$ is in the Hayman class, and thus strongly Gaussian, and, therefore,  amenable to  the Hayman asymptotic formula, \eqref{eq:asymptotic formula coeffs}.

It turns out that these same conditions on $g$ are enough for $f$ being uniformly Hayman; this is the content of Theorem \ref{teor:exponentials uniformly Hayman}.

Denote
\begin{equation}\label{definition omegag}
	\omega_g(t)\triangleq\frac{1}{6}(b_1 t+8 b_2 t^2+\frac{9}{2} t^3g^{\prime\prime\prime}(t)) \, , \quad \mbox{for $t \in (0,R)$}\, .
\end{equation}

\begin{theo}\label{teor:exponentials uniformly Hayman} Let $g$ be a nonconstant power series with radius of convergence $R>0$ and nonnegative coefficients.
	
	Assume that  the variance condition is satisfied \begin{equation}\label{eq:variance condition for e^g} \lim_{t \uparrow R} \big(tg^\prime(t) +t^2g^{\prime\prime}(t)\big)=+\infty\,.\end{equation}
	
	Assume further that there is a cut function $h(t)$ satisfying
	\begin{equation}
		\label{eq:arco mayor para f=e^g} \lim_{t \uparrow R}\omega_g(t) h(t)^3=0\,.\end{equation}
	and
	that there are positive functions $U, V$ defined in $(t_0,R)$, for some  $t_0\in (0,R)$, where $U$ takes values in $(0,\pi]$ and $V$ in $(0,\infty)$ and are such that
	\begin{equation}\label{eq:arco menor para f=e^g bis uno}
		\sup\limits_{|\varphi|\ge\omega}\big(\Re g(t e^{\imath \varphi})-g(t)\big)\le -V(t) \, \omega^2\, , \quad \mbox{for $\omega\le U(t)$ and $t \in (t_0,R)$}\, ,
	\end{equation}
	and thus that the cut $h$ is such that
	\begin{equation}\label{eq:arco menor para f=e^g bis dos}
		h(t)\le U(t), \, \mbox{for $t \in (t_0,1)$}\quad \mbox{and} \quad  \lim_{t \uparrow R} \sigma_f(t) \, e^{-V(t)  h(t)^2}=0\, .
	\end{equation}
	then the function $f=e^g$ is  uniformly Hayman.
\end{theo}

\begin{proof} Condition \eqref{eq:variance condition for e^g} is directly condition \eqref{eq:uniformly Hayman variance}.
	
	We shall verify the conditions on the cuts of the definition of uniformly Hayman power series \eqref{eq:uniformly Hayman major} and \eqref{eq:uniformly Hayman minor bis} with cuts $h(n,t)$ given by
	$$h(n,t)= h(t) \, n^{-\beta}$$
	where the parameter $\beta$ satisfies $1/3 < \beta < 1/2$.
	
	\
	
	From the discussion of Section 4 of \cite{K_uno}, we have that
	$$\Big|\ln \E(e^{\imath \theta \breve{X}_t})+\frac{\theta^2}{2}\Big|\le \omega_g(t)\frac{|\theta|^3}{ \sigma^3(t)}\, , \quad \mbox{for $t \in (0,R)$ and $\theta \in \R$}\, ,$$
	and, thus, that
	$$\Big|n\ln \E(e^{\imath \theta \breve{X}_t/\sqrt{n}})+\frac{\theta^2}{2}\Big|\le \omega_g(t)\frac{|\theta|^3}{ \sigma^3(t) \sqrt{n}}\, , \quad \mbox{for $t \in (0,R)$ and $\theta \in \R$}\, .$$
	
	For $|\theta|\le h(n,t) \sigma_f(t) \sqrt{n}$ we deduce that
	$$\Big|n\ln \E(e^{\imath \theta \breve{X}_t/\sqrt{n}})+\frac{\theta^2}{2}\Big|\le \omega_g(t) \, h(t)^3 \, n^{1{-}3\beta}$$
	
	Hypothesis \eqref{eq:arco mayor para f=e^g} on $h(t)$ and the fact that $1-3\beta <0$ gives us
	that
	$$
	\lim\limits_{[n{\to}\infty \vee t{\uparrow}R]} \omega_g(t) \, h(t)^3 \, n^{1{-}3\beta}=0\, ,$$
	and, thus, that condition \eqref{eq:uniformly Hayman major} is satisfied.
	
	\
	
	Since $h(n,t)\le h(t)\le U(t)$, condition \eqref{eq:arco menor para f=e^g bis uno} gives us that
	$$n \sup\limits_{h(n,t) \le |\theta|\le \pi} \big(\Re g(te^{\imath \theta})-g(t)\big)\le -V(t) h(t)^2 n^{1{-}2\beta}\, ,$$
	and so that
	$$\sup\limits_{h(n,t) \le |\theta|\le \pi} \big|\E(e^{\imath \theta X_t})^n\Big|\le \exp\Big(-V(t) h(t)^2 n^{1{-}2\beta}\Big)\,.$$
	
	Since $\lim_{t \uparrow R}V(t) h(t)^2=\infty$, we have, for a certain $t_0 \in (0,R)$ that $V(t) h(t)^2\ge 1$, for $t \in (t_0,R)$, and thus that
	$$V(t) h(t)^2 n^{1{-}2\beta}\ge V(t) h(t)^2 + n^{1{-}2\beta}\,.$$
	
	We deduce that
	$$\sqrt{n}\sigma_f(t)\, \sup\limits_{h(n,t) \le |\theta|\le \pi} \big|\E(e^{\imath \theta X_t})^n\big| \le \sqrt{n}e^{-n^{1{-}2\beta}}\, \sigma_f(t) e^{-V(t) h(t)^2}\, ,$$
	and,  because of hypothesis \eqref{eq:arco menor para f=e^g bis dos} and since $\beta< 1/2$, that condition \eqref{eq:uniformly Hayman minor bis} of the definition of uniformly Hayman is satisfied.
\end{proof}

\

In \cite{K_dos}  a large number of exponentials $f=e^g$ where $g$ is a power series with nonnegative coefficients   which satisfy the conditions of Theorem \ref{teor:exponentials uniformly Hayman} are exhibited. For instance, the egf of the  Bell numbers, or the generating functions $P$ of partitions or $Q$ of partitions into  distinct parts, are actually uniformly Hayman, and also are uniformly Hayman related examples like the egf of sets of pointed sets, the egf of sets of functions or the ogf of plane partitions or of some colored partitions.

\subsection{Exponential of polynomials}\label{section:exponential polynomial uniformly Hayman}  Let $g$ be a polynomial with nonnegative coefficients
$
g(z)=\sum_{n=0}^N b_n z^n$
and of degree $N$, so that $b_N>0$.

Assume that $Q_g=\gcd\{1\le n \le N: b_n>0\}=1$. Then $f=e^g$ is in the Hayman. This is a particular case of a result of Hayman \cite[Theorem X]{Hayman}. See \cite[Proposition 5.1]{K_uno} for a simpler proof of this particular case.

\smallskip

We are going to show next that \textit{$f=e^g$ is actually uniformly Hayman} with an argument similar to the one used to show in \cite[Proposition 5.1]{K_uno} that $f=e^g$ is in the Hayman class.

Observe first that
$$\sigma_f^2(t)=tg^\prime(t)+t^2 g^{\prime\prime}(t)\sim N^2 b_N t^N, \quad \mbox{as $t \to \infty$}\,.$$
Thus, the variance  condition \eqref{eq:uniformly Hayman variance} of being uniformly Hayman is satisfied.

\medskip

We have
$$\omega_g(t)=\frac{1}{6}\big(b_1 t +8 b_2 t^2+\frac{9}{2} t^3 g^{\prime\prime\prime}(t)\big)=O(t^N)\,, \quad \mbox{as $t \to \infty$}\,.$$

For cuts we propose $h(n,t)=h(t) n^{-\beta}=t^{-N\alpha} n^{-\beta}$, with $\alpha, \beta$ in the interval $(1/3, 1/2)$. For concreteness, we take $\alpha=\beta=5/12$.

From the proof of Theorem \ref{teor:exponentials uniformly Hayman} we have that
$$\begin{aligned}\Big|n\ln \E(e^{\imath \theta \breve{X}_t/\sqrt{n}})+\frac{\theta^2}{2}\Big|&\le \omega_g(t) t^{-5N/4} n^{-1/4}\\&=O\big(t^{-N/4} n^{-1/4}\big)\,,\quad \mbox{for $|\theta|\le h(n,t) \sigma_f(t) \sqrt{n}$}\,,\end{aligned}$$
and thus, we see that condition \eqref{eq:uniformly Hayman major} is satisfied.

\medskip

Now,  the proof that $f=e^g$ is in the Hayman class of \cite{K_uno} gives $\eta\in (0,\pi)$ and $t_0>0$, depending on $g$, so that
$$\sup\limits_{|\theta|>\omega} \big(\Re g(te^{\imath \theta})-g(t)\big)\le -C_g \min \big\{t, t^N \omega^2\big\}, \quad \mbox{for $\omega \ge \eta$ and $t>t_0$}\,,$$
for some constant $C_g$ depending on $g$.

Thus for some $t_1>t_0$,  so that $h(t)<\eta$, for $t>t_1$,  we have that
$$\begin{aligned}n \sup\limits_{h(n,t)\le |\theta|\le \pi} \big(\Re g(te^{\imath \theta})-g(t)\big)&\le -C \min\big\{t n, t^N t^{-5N/6} n^{1-5/6}\big\}\\&=-C \min\big\{tn, t^{N/6} n^{1/6}\big\}\,.\end{aligned}$$

With $\delta=\min\{1, N/6\}$, we have,  for $t\ge 1$, that
$$
\min\{tn, t^{N/6} n^{1/6}\}\ge t^\delta n^{1/6}\ge (1/2) (t^\delta+n^{1/6})\,,$$
and thus, since $|\E(e^{\imath \theta X_t})|=e^{\Re g(te^{\imath \theta})-g(t)}$, that
$$\begin{aligned}
	\sqrt{n} \sigma_f(t) \sup\limits_{h(n,t)\le |\theta|\le \pi}\big|\E(e^{\imath \theta X_t})^n\big|&=O\Big(\sqrt{n}t^{N/2}\exp\big(-C \min\{tn, t^{N/6}n^{1/6}\}\big)\Big)\\
	&=O\Big(\sqrt{n} \exp\big(-(C/2) n^{1/6}\big)\Big)\, O\Big(t^{N/2} \exp\big(-(C/2) t^\delta\big)\Big)
	\,,\end{aligned}$$
and, therefore, condition  \eqref{eq:uniformly Hayman minor} is satisfied.

\end{document}